\documentclass[12pt]{article}
\usepackage[a4paper]{geometry}
\usepackage{amsthm,amssymb,amsmath}
\usepackage{graphicx,cite,color}
\usepackage{longtable,pdflscape,booktabs,caption,multicol}
\usepackage[colorlinks=true,citecolor=black,linkcolor=black,urlcolor=blue]{hyperref}

\theoremstyle{plain}
\newtheorem{theorem}{Theorem}
\newtheorem{lemma}[theorem]{Lemma}
\newtheorem{corollary}[theorem]{Corollary}

\newtheorem{theoremx}{Theorem}

\theoremstyle{definition}

\theoremstyle{remark}
\newtheorem{remark}[theorem]{Remark}

\title{Spectral properties of balanced trees and dendrimers}

\author{Ivan Damnjanovi\'c\thanks{The corresponding author.}\\
\small University of Ni\v s, Faculty of Electronic Engineering\\[-0.8ex] 
\small Aleksandra Medvedeva 14, 18115 Ni\v s, Serbia\\
\small\tt ivan.damnjanovic@elfak.ni.ac.rs\\
\and
Slobodan Filipovski\thanks{The second author is supported by the Slovenian Research Agency through the grants P1-0285, J1-1695, J1-9108 and J1-9110, as well as the bilateral project 337-00-21/2020-09/28 between the Slovenian Research Agency and the Ministry of Education, Science and Technological Development of the Republic of Serbia.}\\
\small University of Primorska, Faculty of Mathematics,\\[-0.8ex]
\small Natural Sciences and Information Technologies\\[-0.8ex]
\small Glagolja\v{s}ka 8, 6000 Koper, Slovenia\\
\small\tt slobodan.filipovski@famnit.upr.si\\
\and
Dragan Stevanovi\'c\thanks{The third author is supported by the Serbian Academy of Sciences and Arts through the grant F-159 and by the Ministry of Education, Science and Technological Development of the Republic of Serbia through the Mathematical Institute of SASA, as well as the bilateral project 337-00-21/2020-09/28 between the Slovenian Research Agency and the Ministry of Education, Science and Technological Development of the Republic of Serbia.}\\
\small Mathematical Institute of the Serbian Academy of Sciences and Arts\\[-0.8ex]
\small Kneza Mihaila 36, 11000 Belgrade, Serbia\\
\small\tt dragan\_stevanovic@mi.sanu.ac.rs}

\begin{document}

\maketitle

\newpage
\begin{abstract}
We investigate the spectral properties of balanced trees and dendrimers, with a view toward unifying and improving the existing results. Here we find a semi-factorized formula for their characteristic polynomials. Afterwards, we determine their spectra via the aforementioned factors. In the end, we analyze the behavior of the energy of dendrimers and compute lower and upper bound approximations for it.


\end{abstract}

\section{Introduction}

Let $G$ be a simple graph with $n$ vertices and the adjacency matrix $A(G)$. 
Let $P_x(G)=\det(xI-A(G))$ denote the characteristic polynomial of $A(G)$, and
let $\lambda_1, \lambda_2, \ldots, \lambda_n$ be the eigenvalues of $A(G)$.
The energy of $G$ is defined as $E(G) = \sum_{j=1}^{n} \lvert \lambda_j \rvert$, 
as introduced by Gutman in \cite{gutman_energy}.
Let $\sigma^*(G)$ represent the set of all the distinct eigenvalues of~$A(G)$.

A balanced tree is an unweighted rooted tree such that all the vertices from the same level have an equal degree, while a dendrimer is defined as a balanced tree whose internal vertices all have the same degree. We will use $d(l,k)$ to denote the dendrimer which contains $l+1$ levels enumerated from $0$ to~$l$ with each internal vertex having a degree of~$k$. More about the various applications of dendrimers can be found, for example, in \cite{dendrimer_applications_1, dendrimer_applications_2}. 

In this paper, we investigate the spectral properties of balanced trees and dendrimers. The two main results we obtain deal with the computation of $E(d(l, k))$ and are given as follows:
\begin{theorem}\label{main_th_1}
    For any fixed value of $l \ge 1$, we have
    \begin{equation*}
        E(d(l, k)) \sim 2 (k-1)^{l-1/2} \qquad \mbox{as $\ k \to \infty$} \, .
    \end{equation*}
    Also, for any fixed value of $k \ge 3$, we have
    \begin{equation*}
        E(d(l, k)) \sim \mu_k (k-1)^{l-1/2} \qquad \mbox{as $\ l \to \infty$} \, .
    \end{equation*}
    where $\mu_k$ is the positive real number which represents the sum of the convergent positive series
    $$\sum_{j = 0}^{\infty} f_j (k-1)^{-j}$$
    defined by
    $$f_j = \begin{cases}
        2 \csc\left( \dfrac{\pi}{2j+6} \right) - 2\csc\left( \dfrac{\pi}{2j+2} \right), & 2 \mid j \, ,\\
        2 \cot\left( \dfrac{\pi}{2j+6} \right) - 2\cot\left( \dfrac{\pi}{2j+2} \right), & 2 \nmid j \, .
    \end{cases}$$
\end{theorem}
\begin{theorem}\label{main_th_2}
    For a given dendrimer $d(l, k)$, where $k \ge 3$ and $l \ge 2$, we have
    \begin{align*}
        E(d(l,k)) &< (k-1)^{l-1/2} \left( 2 + \dfrac{0.5 + \sqrt{2} + \sqrt{3} + \sqrt{5}}{k-1} \right) ,\\
        E(d(l,k)) &> (k-1)^{l-1/2} \left( 2 + \dfrac{2 \sqrt{2}}{k-1} \right) .
    \end{align*}
\end{theorem}

The remainder of this paper is structured as follows. In Section~\ref{sc-balanced} we prove a result which yields the characteristic polynomials and spectra of balanced trees, improving the previous result of Rojo and Soto~\cite{rojo}. We further focus on computing the characteristic polynomials and spectra of dendrimers in Sections \ref{sc-dendrimer-1}--\ref{sc-dendrimer-2}. The results in these sections improve the previous results of Stevanovi\'c~\cite{stevanovic_energy} and Bokhary and Tabassum~\cite{bokhary_energy}. In Section~\ref{sc-dendrimer-3} we deal with the computation of the energy of dendrimers and use the results from Sections \ref{sc-dendrimer-1}--\ref{sc-dendrimer-2} in order to prove the two main theorems, namely Theorem \ref{main_th_1} and Theorem \ref{main_th_2}.

\section{Characteristic polynomials and eigenvalues of balanced trees}
\label{sc-balanced}

Let $T$ be a balanced tree. First, we can uniquely represent $T$ up to isomorphism by assigning to it a characteristic tuple of positive integers $C_T = (c_1, c_2, \ldots, c_l)$. The number of elements in this tuple dictates how many internal levels $T$ has. We will enumerate the tree levels from $0$ to $l$, where the levels $0, 1, \ldots, l-1$ are internal, while the level $l$ contains the leaves. In the tuple $C_T$, the element $c_j$ describes the number of children of each vertex in level $j-1$. In other words, $c_1$ represents the degree of the root of $T$, while $c_j + 1$ determines the degree of each vertex in level $j-1$, for all $2 \le j \le l$.

Furthermore, we are going to use $n_j$ to denote the total number of vertices in level $j$, for $0 \le j \le l$. Here we obviously have that $n_0 = 1$ and $n_j = n_{j-1} c_j$ for all $1 \le j \le l$. Also, let $n_T = \sum_{j=0}^{l}n_j$ so that $n_T$ represents the total number of vertices in $T$. For convenience, we will define $n_{-1} = 0$.

The following theorem describes the characteristic polynomials of arbitrary balanced trees.

\begin{theorem}\label{main-balanced}
Let $T$ be a balanced tree such that its characteristic tuple equals $C_T = (c_1, c_2, \ldots, c_l)$. If we define a sequence of polynomials $Q_0(x), Q_1(x), \ldots, \linebreak Q_{l+1}(x)$ by
\begin{align*}
    Q_0(x) &= 1 \, ,\\
    Q_1(x) &= x \, ,\\
    Q_{j+2}(x) &= x \, Q_{j+1}(x) - c_{l-j} \, Q_j(x) \qquad \mbox{for all $\, 0 \le j \le l-1$} \, ,
\end{align*}
then
\begin{equation*}\label{main-formula}
    P_x(T) = \prod_{j = 1}^{l+1} Q_j(x)^{n_{l+1-j}-n_{l-j}} \, .
\end{equation*}
\end{theorem}

Before we prove the theorem, we will make certain preliminary definitions which will aid us in creating a more concise proof. For two given positive integers $\alpha$ and $\beta$ such that $\beta \mid \alpha$, we will use $B_{\alpha, \beta} = (b_{ij})_{\alpha \times \beta} \in \mathbb{R}^{\alpha \times \beta}$ to denote the binary matrix whose rows and columns are enumerated from $0$ to $\alpha-1$ and $0$ to $\beta-1$ respectively, such that $b_{ij} = 1$ if and only if $\lfloor \frac{i\beta}{\alpha} \rfloor = j$. For example, we have
\begin{align*}
    B_{8, 2} = \begin{bmatrix}
     1 & 0\\
     1 & 0\\
     1 & 0\\
     1 & 0\\
     0 & 1\\
     0 & 1\\
     0 & 1\\
     0 & 1
    \end{bmatrix} , &&
    B_{9, 3} = \begin{bmatrix} 
    1 & 0 & 0\\
    1 & 0 & 0\\
    1 & 0 & 0\\
    0 & 1 & 0\\
    0 & 1 & 0\\
    0 & 1 & 0\\
    0 & 0 & 1\\
    0 & 0 & 1\\
    0 & 0 & 1
    \end{bmatrix} , &&
    B_{10, 5} = \begin{bmatrix}
    1 & 0 & 0 & 0 & 0\\
    1 & 0 & 0 & 0 & 0\\
    0 & 1 & 0 & 0 & 0\\
    0 & 1 & 0 & 0 & 0\\
    0 & 0 & 1 & 0 & 0\\
    0 & 0 & 1 & 0 & 0\\
    0 & 0 & 0 & 1 & 0\\
    0 & 0 & 0 & 1 & 0\\
    0 & 0 & 0 & 0 & 1\\
    0 & 0 & 0 & 0 & 1
    \end{bmatrix} .
\end{align*}
We can concisely express the adjacency matrix of $T$ with the help of these matrices:
$$A(T) = \begin{bmatrix}
    \mathbf{O} & B_{n_l, n_{l-1}} & \mathbf{O} & \cdots & \mathbf{O} & \mathbf{O}\\
    B_{n_l, n_{l-1}}^T & \mathbf{O} & B_{n_{l-1}, n_{l-2}} & \cdots & \mathbf{O} & \mathbf{O}\\
    \mathbf{O} & B_{n_{l-1}, n_{l-2}}^T & \mathbf{O} & \cdots & \mathbf{O} & \mathbf{O}\\
    \vdots & \vdots & \vdots & \ddots & \vdots & \vdots\\
    \mathbf{O} & \mathbf{O} & \mathbf{O} & \cdots & \mathbf{O} & B_{n_1, n_0}\\
    \mathbf{O} & \mathbf{O} & \mathbf{O} & \cdots & B_{n_1, n_0}^T & \mathbf{O}
\end{bmatrix} \, .$$
Together with $P_x(T) = \det(x \, \mathbf{I}_n - A(T))$, this leads us to
\begin{equation}\label{char_pol}
P_x(T) = \begin{vmatrix}
    x\,\mathbf{I}_{n_l} & -B_{n_l, n_{l-1}} & \mathbf{O} & \cdots & \mathbf{O} & \mathbf{O}\\
    -B_{n_l, n_{l-1}}^T & x\,\mathbf{I}_{n_{l-1}} & -B_{n_{l-1}, n_{l-2}} & \cdots & \mathbf{O} & \mathbf{O}\\
    \mathbf{O} & -B_{n_{l-1}, n_{l-2}}^T & x\,\mathbf{I}_{n_{l-2}} & \cdots & \mathbf{O} & \mathbf{O}\\
    \vdots & \vdots & \vdots & \ddots & \vdots & \vdots\\
    \mathbf{O} & \mathbf{O} & \mathbf{O} & \cdots & x\,\mathbf{I}_{n_1} & -B_{n_1, n_0}\\
    \mathbf{O} & \mathbf{O} & \mathbf{O} & \cdots & -B_{n_1, n_0}^T & x\,\mathbf{I}_{n_0}
\end{vmatrix}  .
\end{equation}
We will compute $P_x(T)$ from Eq.~(\ref{char_pol}).

{\em Proof of Theorem \ref{main-balanced}}. 
By multiplying the $(j+1)$-th block row in Eq.~(\ref{char_pol}) with $Q_j(x)$, for all $1 \le j \le l$, we obtain
\begin{align*}
& P_x(T) \prod_{j = 1}^{l} Q_j(x)^{n_{l-j}} =\\
    &=
    \begin{vmatrix}
    x\,Q_0(x)\mathbf{I}_{n_l} & -Q_0(x)B_{n_l, n_{l-1}} & \mathbf{O} & \cdots & \mathbf{O}\\
    -Q_1(x)B_{n_l, n_{l-1}}^T & x\,W_1(x)\mathbf{I}_{n_{l-1}} & -Q_1(x)B_{n_{l-1}, n_{l-2}} & \cdots & \mathbf{O}\\
    \mathbf{O} & -Q_2(x)B_{n_{l-1}, n_{l-2}}^T & x\,Q_2(x)\mathbf{I}_{n_{l-2}} & \cdots & \mathbf{O}\\
    \vdots & \vdots & \vdots & \ddots & \vdots &\\
    \mathbf{O} & \mathbf{O} & \mathbf{O} & \cdots & x\,Q_l(x)\mathbf{I}_{n_0}
\end{vmatrix} \, .
\end{align*}
The determinant on the right-hand side can easily be calculated by applying the Gaussian elimination on the given block matrix. If we multiply the first block row by $B_{n_l, n_{l-1}}^T$ to the left and then add the obtained result to the second row, we get
\begin{align*}
& P_x(T) \prod_{j = 1}^{l} Q_j(x)^{n_{l-j}} =\\
    &=
    \begin{vmatrix}
    x\,Q_0(x)\mathbf{I}_{n_l} & -Q_0(x)B_{n_l, n_{l-1}} & \mathbf{O} & \cdots & \mathbf{O}\\
    \mathbf{O} & Q_2(x)\mathbf{I}_{n_{l-1}} & -Q_1(x)B_{n_{l-1}, n_{l-2}} & \cdots & \mathbf{O}\\
    \mathbf{O} & -Q_2(x)B_{n_{l-1}, n_{l-2}}^T & x\,Q_2(x)\mathbf{I}_{n_{l-2}} & \cdots & \mathbf{O}\\
    \vdots & \vdots & \vdots & \ddots & \vdots &\\
    \mathbf{O} & \mathbf{O} & \mathbf{O} & \cdots & x\,Q_l(x)\mathbf{I}_{n_0}
\end{vmatrix} \, .
\end{align*}
We can then multiply the second block row by $B_{n_{l-1}, n_{l-2}}^T$ to the left and add the result to the third row, and so on and so forth, until we reach
\begin{align*}
& P_x(T) \prod_{j = 1}^{l} Q_j(x)^{n_{l-j}} =\\
    &=
    \begin{vmatrix}
    Q_1(x)\mathbf{I}_{n_l} & -Q_0(x)B_{n_l, n_{l-1}} & \mathbf{O} & \cdots & \mathbf{O}\\
    \mathbf{O} & Q_2(x)\mathbf{I}_{n_{l-1}} & -Q_1(x)B_{n_{l-1}, n_{l-2}} & \cdots & \mathbf{O}\\
    \mathbf{O} & \mathbf{O} & Q_3(x)\mathbf{I}_{n_{l-2}} & \cdots & \mathbf{O}\\
    \vdots & \vdots & \vdots & \ddots & \vdots &\\
    \mathbf{O} & \mathbf{O} & \mathbf{O} & \cdots & Q_{l+1}(x)\mathbf{I}_{n_0}
\end{vmatrix} \, ,
\end{align*}
which finally gives us
$$P_x(T) \prod_{j = 1}^{l} Q_j(x)^{n_{l-j}} = \prod_{j = 1}^{l+1}Q_j(x)^{n_{l+1-j}} \, .$$
By taking into consideration that $n_{-1} = 0$, we get that for any complex number~$z$ which is not a root of any of the polynomials $Q_1(x), Q_2(x), \ldots, Q_l(x)$, the following equality must hold:
\begin{equation}\label{poly_equal}
P_z(T) = \prod_{j = 1}^{l+1} Q_j(z)^{n_{l+1-j} - n_{l-j}} \, .
\end{equation}
Since each polynomial $Q_j(x)$ is of degree $j$, for all $1 \le j \le l$, we conclude that there are infinitely many complex numbers which are not a root of any of these polynomials. From Eq.\ (\ref{poly_equal}), we see that the polynomials $P_x(T)$ and $\prod_{j = 1}^{l+1} Q_j(x)^{n_{l+1-j} - n_{l-j}}$ must be equal in infinitely many points, which implies that these two polynomials are identical.
\hfill\qed

By Theorem \ref{main-balanced}, $P_x(T)$ is a product of those $Q_j(x)$ for which $n_{l+1-j} - n_{l-j}$ is positive.
This implies the following result.

\begin{theorem}\label{main-corollary}
Let $\Phi$ be the set of positive integers $1 \le j \le l+1$ such that $n_{l+1-j} > n_{l-j}$. 
Then we have
$$
\sigma^*(T) = \bigcup_{j \in \Phi} \, \{ x \colon Q_j(x) = 0\}.
$$
\end{theorem}

\begin{remark}
The condition $n_{l+1-j} > n_{l-j}$ is always satisfied for $j = l+1$. For $j \le l$, it becomes equivalent to $c_{l+1-j} \neq 1$.
\end{remark}

The use of Theorem~\ref{main-balanced} is exemplified in the following section on the case of dendrimers.

\section{Characteristic polynomials of dendrimers}
\label{sc-dendrimer-1}

If we view a dendrimer $d(l, k)$ as a balanced tree, we see that its characteristic tuple equals $(k, \underbrace{k-1, k-1, \ldots, k-1}_{\text{ $l-1$ times}})$. We can now apply Theorem \ref{main-balanced} in order to compute the characteristic polynomial of $d(l, k)$.

\begin{theorem}\label{dendrimer-char}
Let $W_{l, 0}(x, k), W_{l, 1}(x, k), \ldots, W_{l, l+1}(x, k)$ be a sequence of polynomials defined in the following manner
\begin{align}
\begin{split}\label{w_polynomials}
    W_{l, 0}(x, k) &= 1 \, ,\\
    W_{l, 1}(x, k) &= x \, ,\\
    W_{l, j}(x, k) &= x \, W_{l, j-1}(x, k) - (k-1) \, W_{l, j-2}(x, k) \qquad \mbox{for all $\, 2 \le j \le l$} \, ,\\
    W_{l, l+1}(x, k) &= x \, W_{l,l}(x, k) - k \, W_{l, l-1}(x, k) \, .
\end{split}
\end{align}
Then, for all $k \ge 2$ and $l \ge 1$,
\begin{equation}\label{dendrimer-formula}
P_x(d(l, k)) = W_{l, l+1}(x, k) W_{l,l}(x, k)^{k-1} \prod_{j = 1}^{l-1} W_{l, j}(x, k)^{k(k-2)(k-1)^{l-1-j}} \, .
\end{equation}
\end{theorem}
\begin{proof}
Given the fact that $c_1 = k$ and $c_j = k-1$ for all $2 \le j \le l$, Theorem \ref{main-balanced} gives us
$$P_x(T) = \prod_{j = 1}^{l+1} W_{l, j}(x, k)^{n_{l+1-j}-n_{l-j}} \, .$$
Since $n_{-1} = 0, n_0 = 1$ and $n_j = k(k-1)^{j-1}$ for all $1 \le j \le l$, we have $n_0 - n_{-1} = 1, n_1 - n_0 = k-1$ and $n_{j+1} - n_j = k(k-1)^j - k(k-1)^{j-1}$ for all $1 \le j \le l-1$. This leads us to
\begin{align*}
    P_x(T) &= W_{l, l+1}(x, k)^{n_0 - n_{-1}} W_{l, l}(x, k)^{n_1 - n_0} \prod_{j = 1}^{l-1} W_{l, j} (x, k)^{n_{l+1-j}-n_{l-j}}\\
    &= W_{l, l+1}(x, k) W_{l, l}(x, k)^{k-1} \prod_{j = 1}^{l-1} W_{l, j}(x, k)^{k(k-1)^{l-j} - k(k-1)^{l-j-1}}\\
    &= W_{l, l+1}(x, k) W_{l, l}(x, k)^{k-1} \prod_{j = 1}^{l-1} W_{l, j}(x, k)^{k(k-1)^{l-j-1}(k-2)} \, .
\end{align*}
\end{proof}

It is relatively easy to use Theorem \ref{dendrimer-char} to obtain the characteristic polynomial of the dendrimer $d(l, k)$,
where $k \ge 2$ can be treated as an integer parameter, while $l \ge 1$ is some fixed and preferably small, positive integer. 
For example, if we put $l = 1$, we will get a sequence of polynomials
\begin{align*}
    W_{1, 0}(x, k) &= 1 \, ,\\
    W_{1, 1}(x, k) &= x \, ,\\
    W_{1, 2}(x, k) &= x^2 - k \, ,
\end{align*}
which quickly gives us the characteristic polynomial
\begin{equation}\label{l_is_1}
    P_x(d(1, k)) = (x^2 - k) x^{k-1} \, .
\end{equation}
Similarly, by setting $l = 2$ we obtain the sequence of polynomials
\begin{align*}
    W_{2, 0}(x, k) &= 1 \, ,\\
    W_{2, 1}(x, k) &= x \, ,\\
    W_{2, 2}(x, k) &= x^2 - (k-1) \, ,\\
    W_{2, 3}(x, k) &= x^3 - (2k-1)x \, ,
\end{align*}
which leads to the characteristic polynomial
\begin{align}
    \nonumber P_x(d(2, k)) &= [x^3-(2k-1)x][x^2-(k-1)]^{k-1} x^{k(k-2)}\\
    \nonumber &= x^{k^2 - 2k + 1} (x^2 - k + 1)^{k-1} (x^2 - 2k + 1)\\
    \label{l_is_2} &= x^{(k-1)^2} (x^2 - k + 1)^{k-1} (x^2 - 2k + 1) \, .
\end{align}
Further for $l=3$ we get the sequence of polynomials
\begin{align*}
    W_{3, 0}(x, k) &= 1 \, ,\\
    W_{3, 1}(x, k) &= x \, ,\\
    W_{3, 2}(x, k) &= x^2 - (k-1) \, ,\\
    W_{3, 3}(x, k) &= x^3 - (2k-2)x \, ,\\
    W_{3, 4}(x, k) &= x^4 - (3k-2)x^2 + k(k-1) \, ,
\end{align*}
which yields the characteristic polynomial
\begin{align*}
    P_x(d(3, k)) ={} & [x^4 - (3k-2)x^2 + k(k-1)][x^3 - (2k-2)x]^{k-1}\\
    			  & [x^2 - (k-1)]^{k(k-2)} x^{k(k-2)(k-1)}\\
		    ={} & [x^4 - (3k-2)x^2 + k(k-1)](x^2 - 2k + 2)^{k-1}\\
		           & (x^2 - k + 1)^{k(k-2)} x^{k(k-2)(k-1) + (k-1)}\\
		    ={} & [x^4 - (3k-2)x^2 + k(k-1)](x^2 - 2k + 2)^{k-1}\\
                              & (x^2 - k + 1)^{k(k-2)} x^{(k-1)^3}.
\end{align*}

\begin{remark}
Note that \cite[Theorem 3.5]{bokhary_energy} gives a slightly inaccurate expression for $P_x(d(3,k))$,
in which the correct factor $x^4 - (3k-2)x^2 + k(k-1)$ above is replaced by an incorrect factor $x^4 - 2(k+1)x^2 + 4(k-1)$. 
\end{remark}

\section{Eigenvalues of dendrimers}
\label{sc-dendrimer-2}

Theorem \ref{dendrimer-char} can be used to determine the eigenvalues of the given dendrimer $d(l, k)$. If $k > 2$, then a real number belongs to $\sigma^*(d(l, k))$ if and only if it represents a root of some polynomial from the sequence $W_{l, 1}(x, k), W_{l, 2}(x, k), \ldots, W_{l, l+1}(x, k)$, given the fact that $c_j > 1$ for all $1 \le j \le l$. A special case occurs when $k = 2$, since this would lead to $c_2 = c_3 = \cdots = c_l = 1$. In this situation, we would have that $\sigma^*(d(l, k))$ is composed solely of real numbers which are a root of $W_{l, l+1}(x, k)$ or $W_{l, l}(x, k)$.
\begin{theorem}\label{dendrimer-spectrum}
    For any $k \ge 3$ and $l \ge 1$, we have
    $$
    \sigma^*(d(l, k)) = \bigcup_{j = 1}^{l+1} \, \{ x \colon W_{l, j}(x, k) = 0\}.
    $$
    If $k = 2$ and $l \ge 1$, then we get
    $$
    \sigma^*(d(l, 2)) = \{ x \colon W_{l, l+1}(x, k) = 0 \vee W_{l, l}(x, k) = 0\}.
    $$
\end{theorem}

A natural question is whether these results can be made explicit. 
The answer to this question is affirmative, as previously shown by one of the present authors in~\cite{stevanovic_energy}. 
Let $E_0(x, a), E_1(x, a), E_2(x, a), \ldots$ be a sequence of polynomials defined via the recurrence relation
\begin{align}
\begin{split}\label{dickson_def}
    E_0(x, a) &= 1 \, ,\\
    E_1(x, a) &= x \, ,\\
    E_j(x, a) &= x \, E_{j-1}(x, a) - a \, E_{j-2}(x, a) \qquad \mbox{for all $\, j \ge 2$} \, .\\
\end{split}
\end{align}
We shall call these polynomials the Dickson polynomials of the second kind, as done so in \cite[pp.\ 9--10]{dickson_polynomials}. 
By comparing Eqs.\ (\ref{dickson_def}) and~(\ref{w_polynomials}), 
we see that the polynomials $W_{l, 0}(x, k), W_{l, 1}(x, k), \ldots, W_{l, l}(x, k)$ actually represent the Dickson polynomials of the second kind $E_{0}(x, k-1), E_{1}(x, k-1), \ldots, E_{l}(x, k-1)$, respectively. This means that the polynomial $W_{l, j}(x, k)$ must have $j$ distinct simple roots 
$2\sqrt{k-1}\cos\left(\frac{1}{j+1}\pi\right), 
  2\sqrt{k-1}\cos\left(\frac{2}{j+1}\pi\right), 
  \ldots, 
  2\sqrt{k-1}\cos\left(\frac{j}{j+1}\pi\right)$, 
for each $0 \le j \le l$ (see, for example, \cite[pp.\ 9--10]{dickson_polynomials}).

Note that the roots of $W_{l, l+1}(x, k)$ cannot be found as easily, given the fact that this polynomial is not a Dickson polynomial of the second kind, unlike all of its predecessors. This irregularity occurs due to the fact that $W_{l, l+1}(x, k)$ is defined via a recurrence relation which is slightly different from all of the previous ones.

By taking these facts into consideration, we reach the following theorem.

\begin{theorem}\label{spectra}
    For any $k \ge 3$ and $l \ge 1$, we have
    \begin{align*}
    \sigma^*(d(l, k)) ={} & \left\{ 2\sqrt{k-1} \cos\left(\frac{h}{j+1}\pi\right) \colon 1 \le h \le j \le l \right\}\\
    & \cup \, \{ x \in \mathbb{R} \colon W_{l, l+1}(x, k) = 0 \} \, .
    \end{align*}
    On the other hand, if $k = 2$ and $l \ge 1$, then
    $$\sigma^*(d(l, 2)) = \left\{ 2\cos\left(\frac{h}{2l+2}\pi\right) \colon 1 \le h \le 2l+1 \right\} \, .$$
\end{theorem}
\begin{proof}
    The dendrimer $d(l, 2)$ is just a path graph composed of $2l+1$ vertices. 
    It is known (see, for example, \cite[pp.\ 18]{spectra_of_graphs}) that the spectrum of this graph is composed of the real numbers $2\cos\left(\frac{1}{2l+2}\pi\right), 2\cos\left(\frac{2}{2l+2}\pi\right), \ldots,  2\cos\left(\frac{2l+1}{2l+2}\pi\right)$. 
    
    For $k \ge 3$, from $W_{l, j}(x, k) \equiv E_j(x, k-1)$ for $0 \le j \le l$, we have
    \begin{equation}\label{some_sets}
        \bigcup_{j = 1}^{l} \, \{ x \in \mathbb{R} \colon W_j(x) = 0 \} = \left\{ 2\sqrt{k-1} \cos\left(\frac{h}{j+1}\pi\right) \colon 1 \le h \le j \le l \right\} .
    \end{equation}
\end{proof}

Despite the fact that the roots of $W_{l, l+1}(x, k)$ cannot explicitly be found, a relatively good approximation can be made. Moreover, $W_{l, l+1}(x, k)$ represents a Geronimus polynomial of degree $l+1$ (see, for example, \cite{geronimus}). This fact can be shown if we notice that
\begin{align*}
    W_{l, l+1}(x, k) &= x \, W_{l, l}(x, k) - k \, W_{l, l-1}(x, k)\\
    &= x \, E_l(x, k-1) - (k-1) \, E_{l-1}(x, k-1) - E_{l-1}(x, k-1)\\
    &= E_{l+1}(x, k-1) - E_{l-1}(x, k-1) \, ,
\end{align*}
for each $l \ge 1$.

For a fixed parameter $a \in \mathbb{R}$, we will use $G_j(x, a)$ to denote the sequence of Geronimus polynomials defined via
\begin{align*}
    G_0(x, a) &= 1 \, ,\\
    G_1(x, a) &= x \, ,\\
    G_2(x, a) &= x^2 - a \, ,\\
    G_j(x, a) &= x G_{j-1}(x, a) - (a-1)G_{j-2}(x, c) \qquad \mbox{for all $\, j \ge 3$} \, .
\end{align*}
It is easily proven via mathematical induction that $G_j(x, a) = E_j(x, a-1) - E_{j-2}(x, a-1)$ for each $j \ge 2$, which ultimately shows that $W_{l, l+1}(x, k) \equiv G_{l+1}(x, k)$. This observation leads to the following lemma.

\begin{lemma}\label{geronimus_approx}
    For $k \ge 3$, the polynomial $W_{l, l+1}(x, k)$ has $l+1$ simple real roots determined by the set
    $$\{ \pm \alpha_1, \pm \alpha_2, \ldots, \pm \alpha_{\lfloor (l+1)/2 \rfloor} \}$$
    if $l+1$ is even, or
    $$\{ \pm \alpha_1, \pm \alpha_2, \ldots, \pm \alpha_{\lfloor (l+1)/2 \rfloor} \} \cup \{ 0 \}$$
    if $l+1$ is odd, where the numbers $\alpha_1, \alpha_2, \ldots, \alpha_{\lfloor (l+1)/2 \rfloor}$ form a strictly decreasing positive sequence which satisfies
    $$ 2 \sqrt{k-1} \cos \left( \dfrac{j-0.5}{l+2} \pi \right) > \alpha_j > 2 \sqrt{k-1} \cos \left( \dfrac{j+0.5}{l+2} \pi \right)$$
    for each $j \in \{1, 2, 3, \ldots, \lfloor (l+1)/2 \rfloor \}$.
\end{lemma}
\begin{proof}
     Given the fact that $W_{l, l+1}(x, k) \equiv G_{l+1}(x, k)$, we know from \cite{geronimus} that for $k \ge 3$ the polynomial $W_{l, l+1}(x, k)$ must have $l+1$ distinct real roots $\alpha_1, \alpha_2, \ldots, \alpha_{l+1}$ such that
     \begin{equation}\label{roots_geronimus}
     2 \sqrt{k-1} \cos \left( \dfrac{j-0.5}{l+2} \pi \right) > \alpha_j > 2 \sqrt{k-1} \cos \left( \dfrac{j+0.5}{l+2} \pi \right)
     \end{equation}
     for each $j \in \{1, 2, 3, \ldots, l+1 \}$. Since the cosine is a strictly decreasing function on $[0, \pi]$, Eq.\ (\ref{roots_geronimus}) immediately shows that
     $$ \alpha_1 > \alpha_2 > \alpha_3 > \cdots > \alpha_l > \alpha_{l+1} \, .$$
     Also, due to the fact that the cosine takes positive values on $[0, \pi/2)$ and negative values on $(\pi/2, \pi]$, we obtain
     \begin{equation}\label{pos_ger}
     \alpha_1 > \alpha_2 > \cdots > \alpha_{(l+1)/2} > 0 > \alpha_{(l+1)/2+1} > \cdots > \alpha_l > \alpha_{l+1}
     \end{equation}
     if $l+1$ is even, and
     \begin{equation}\label{neg_ger}
     \alpha_1 > \alpha_2 > \cdots > \alpha_{l/2} > 0 > \alpha_{l/2+2} > \cdots > \alpha_l > \alpha_{l+1}
     \end{equation}
     if $l+1$ is odd.
     
     It is trivial to prove that $G_j(x, k)$ is an even polynomial for each even $j$ and an odd polynomial for each odd $j$. If $l+1$ is even, this observation and Eq.\ (\ref{pos_ger}) together give that $\alpha_{l+2-j} = -\alpha_j$ for each $1 \le j \le (l+1)/2$. This implies that the roots of $W_{l, l+1}(x, k)$ are described via set $\{ \pm \alpha_1, \pm \alpha_2, \ldots, \pm \alpha_{\lfloor (l+1) / 2 \rfloor} \}$, where $\alpha_1 > \alpha_2 > \cdots > \alpha_{\lfloor (l+1) / 2 \rfloor} > 0$.
     
     If $l+1$ is odd, then $0$ must be a root of $W_{l, l+1}(x, k)$, which implies $\alpha_{l/2+1} = 0$. On the other hand, Eq.\ (\ref{neg_ger}) gives $\alpha_{l+2-j} = -\alpha_j$ for each $1 \le j \le l/2$. We conclude that in this case the roots of $W_{l, l+1}(x, k)$ can be described by the set $\{ \pm \alpha_1, \pm \alpha_2, \ldots, \pm \alpha_{\lfloor (l+1) / 2 \rfloor} \} \cup \{ 0 \}$, where $\alpha_1 > \alpha_2 > \cdots > \alpha_{\lfloor (l+1) / 2 \rfloor} > 0$.
\end{proof}

Lemma \ref{geronimus_approx} shall prove to be quite useful while approximating the energy of dendrimers in the following section.

\section{Energy of dendrimers}
\label{sc-dendrimer-3}

Here we discuss the computation of the energy of dendrimers. Due to the fact that the roots of the Geronimus polynomials cannot explicitly be found, we are unable to obtain an exact expression for the energy of a dendrimer $d(l, k)$ whenever $k \ge 3$. Instead, we compute a reasonable approximation of $E(d(l, k))$ that covers all of the dendrimers whenever $k \ge 3$ and $l \ge 1$, together with the exact value of $E(d(l, 2))$ for each $l \ge 1$. Afterwards, we inspect the asymptotic behavior of $E(d(l, k))$ when analyzed as a function of two variables $l \in \mathbb{N}$ and $k \in \mathbb{N} \setminus \{1, 2\}$. In the end, we give another approximative formula of $E(d(l, k))$ when $k \ge 3$, which reflects the asymptotic behavior of the function. To begin with, we state the following theorem:

\allowdisplaybreaks
\begin{theorem}\label{energy_approx}
    For any $k \ge 3$, the energy of the dendrimer $d(l, k)$ can be approximated in the following manner:
    \begin{align}\label{complicated_l3_1}
        E(d(l, k)) &< \sum_{j = 0}^{l-1} f_j (k-1)^{(l-1/2)-j} + 2(k-1)^{1/2} \, ,\\
        \label{complicated_l3_2}
        E(d(l, k)) &> \sum_{j = 0}^{l-1} f_j (k-1)^{(l-1/2)-j} - 2.4(k-1)^{1/2} \, ,
    \end{align}
    where
    $$f_j = \begin{cases}
        2 \csc\left( \dfrac{\pi}{2j+6} \right) - 2\csc\left( \dfrac{\pi}{2j+2} \right), & 2 \mid j \, ,\\
        2 \cot\left( \dfrac{\pi}{2j+6} \right) - 2\cot\left( \dfrac{\pi}{2j+2} \right), & 2 \nmid j \, .
    \end{cases}$$
    Also, the energy of the dendrimer $d(l, 2)$ can be computed via:
    \begin{equation}\label{simple_l2}
        E(d(l, 2)) = 2 \left( \cot\left( \dfrac{\pi}{4l+4} \right) - 1\right) .
    \end{equation}
\end{theorem}

In order to make the proof of Theorem \ref{energy_approx} easier to follow, we shall introduce and prove two auxiliary lemmas. First of all, let $\Psi(W_{l, j}(x, k))$ denote the sum of absolute values of all the roots of the polynomial $W_{l, j}(x, k)$.

\begin{lemma}\label{all_wj_exact}
    For each $0 \le j \le l$, we have
    $$\Psi(W_{l, j}(x, k)) = 2 \sqrt{k-1} \left( \cot\left( \dfrac{\pi}{2j+2} \right) - 1 \right)$$
    if $j$ is odd, and
    $$\Psi(W_{l, j}(x, k)) = 2 \sqrt{k-1} \left( \csc\left( \dfrac{\pi}{2j+2} \right) - 1 \right)$$
    if $j$ is even.
\end{lemma}
\begin{proof}
    The equality is directly proven for $j = 0$. The polynomial $W_{l, 0}(x, k) = 1$ is of degree $0$ and has no roots, hence $\Psi(W_{l, 0}(x, k)) = 0$, while the according right-hand side $2 \sqrt{k-1} \left( \csc\left( \dfrac{\pi}{2} \right) - 1 \right)$ also obviously equals $0$.

    Suppose $1 \le j \le l$. We know that the roots of the polynomial $W_{l, j}(x, k)$ must be $2\sqrt{k-1}\cos\left(\frac{1}{j+1}\pi\right), 2\sqrt{k-1}\cos\left(\frac{2}{j+1}\pi\right), \ldots, 2\sqrt{k-1}\cos\left(\frac{j}{j+1}\pi\right)$. Hence
    \begin{align*}
        \Psi(W_{l, j}(x, k)) &= \sum_{h = 1}^{j}\left| 2\sqrt{k-1} \cos\left(\frac{h}{j+1}\pi\right) \right|\\
        &= 2\sqrt{k-1} \sum_{h = 1}^{j}\left|\cos\left(\frac{h}{j+1}\pi\right) \right| .
    \end{align*}
    Since $\cos\left(\frac{h}{j+1}\pi\right) > 0$ for $1 \le h < \frac{j+1}{2}$ and $\cos\left(\frac{h}{j+1}\pi\right) = -\cos\left(\frac{j+1-h}{j+1}\pi\right)$ for all $\frac{j+1}{2} < h \le j$, we can rewrite the last expression as
    $$\Psi(W_{l, j}(x, k)) = 4\sqrt{k-1} \sum_{h = 1}^{\lfloor\frac{j}{2}\rfloor}\cos\left(\frac{h}{j+1}\pi\right) .$$
    Let us denote $\zeta = e^{\frac{i \pi}{j+1}}$. It is convenient to replace $\cos\left(\frac{h}{j+1}\pi\right)$ with $\frac{\zeta^h + \zeta^{-h}}{2}$. This gives us:
    \begin{align*}
        \Psi(W_{l, j}(x, k)) &= 4\sqrt{k-1} \sum_{h = 1}^{\lfloor\frac{j}{2}\rfloor} \frac{\zeta^h + \zeta^{-h}}{2}\\
        &= 2\sqrt{k-1} \sum_{h = 1}^{\lfloor\frac{j}{2}\rfloor} \left(\zeta^h + \zeta^{-h}\right)\\
        &= 2\sqrt{k-1} \left( \left( \sum_{h = -\lfloor\frac{j}{2}\rfloor}^{\lfloor\frac{j}{2}\rfloor}\zeta^h \right) - 1 \right) .
    \end{align*}
    Since $\zeta \neq 1$, we can use the standard formula for summing a geometric progression in order to get
    \begin{align*}
        \Psi(W_{l, j}(x, k)) &= 2\sqrt{k-1} \left( \dfrac{\sum_{h = 0}^{2 \lfloor\frac{j}{2}\rfloor}\zeta^h}{\zeta^{\lfloor\frac{j}{2}\rfloor}} - 1 \right)\\
        &= 2\sqrt{k-1} \left( \dfrac{\zeta^{2 \lfloor\frac{j}{2}\rfloor + 1} - 1}{\zeta^{\lfloor\frac{j}{2}\rfloor}(\zeta-1)} - 1 \right)\\
        &= 2\sqrt{k-1} \left( \dfrac{\zeta^{\lfloor\frac{j}{2}\rfloor + 1} - \zeta^{-\lfloor\frac{j}{2}\rfloor}}{\zeta-1} - 1 \right)\\
        &= 2\sqrt{k-1} \left( \dfrac{\left(\zeta^{\lfloor\frac{j}{2}\rfloor + 1} - \zeta^{-\lfloor\frac{j}{2}\rfloor} \right)\left(\frac{1}{\zeta}-1\right)}{(\zeta-1)\left(\frac{1}{\zeta}-1\right)} -1 \right) .
    \end{align*}
    By taking into consideration that
    \begin{align*}
        \left(\zeta^{\lfloor\frac{j}{2}\rfloor + 1} - \zeta^{-\lfloor\frac{j}{2}\rfloor} \right)\left(\frac{1}{\zeta}-1\right) &= \zeta^{\lfloor\frac{j}{2}\rfloor}+\zeta^{-\lfloor\frac{j}{2}\rfloor} - \zeta^{\lfloor\frac{j}{2}\rfloor+1} - \zeta^{-\lfloor\frac{j}{2}\rfloor-1}\\
        &= 2\cos\left( \frac{\lfloor\frac{j}{2}\rfloor}{j+1}\pi\right)-2\cos\left( \frac{\lfloor\frac{j}{2}\rfloor+1}{j+1}\pi\right)
    \end{align*}
    and $(\zeta-1)\left(\frac{1}{\zeta}-1\right) = 2 - \left(\zeta + \frac{1}{\zeta}\right) = 2 - 2 \cos\left(\frac{1}{j+1}\pi\right)$, we conclude that
    \begin{align*}
        \Psi(W_{l, j}(x, k)) &= 2\sqrt{k-1} \left( \frac{2\cos\left( \frac{\lfloor\frac{j}{2}\rfloor}{j+1}\pi\right)-2\cos\left( \frac{\lfloor\frac{j}{2}\rfloor+1}{j+1}\pi\right)}{2 - 2 \cos\left(\frac{1}{j+1}\pi\right)}-1 \right)\\
        &= 2\sqrt{k-1} \left( \frac{\cos\left( \frac{\lfloor\frac{j}{2}\rfloor}{j+1}\pi\right)-\cos\left( \frac{\lfloor\frac{j}{2}\rfloor+1}{j+1}\pi\right)}{1 - \cos\left(\frac{1}{j+1}\pi\right)}-1 \right) .
    \end{align*}
    If $j$ is odd, then $\lfloor\frac{j}{2}\rfloor = \frac{j-1}{2}$ and $\frac{\lfloor\frac{j}{2}\rfloor+1}{j+1}\pi=\frac{\pi}{2}$, which transforms the given expression into
    \begin{align*}
        \Psi(W_{l, j}(x, k)) &= 2\sqrt{k-1} \left( \frac{\cos\left( \frac{j-1}{j+1}\cdot\frac{\pi}{2}\right)}{1 - \cos\left(\frac{1}{j+1}\pi\right)}-1 \right)\\
        &= 2\sqrt{k-1} \left( \frac{\sin\left( \frac{2}{j+1}\cdot\frac{\pi}{2}\right)}{2 \sin^2\left(\frac{\pi}{2j+2}\right)}-1 \right)\\
        &= 2\sqrt{k-1} \left( \frac{\sin\left( \frac{\pi}{j+1}\right)}{2 \sin^2\left(\frac{\pi}{2j+2}\right)}-1 \right)\\
        &= 2\sqrt{k-1} \left( \frac{2\sin\left( \frac{\pi}{2j+2}\right)\cos\left( \frac{\pi}{2j+2}\right)}{2 \sin^2\left(\frac{\pi}{2j+2}\right)}-1 \right)\\
        &= 2\sqrt{k-1} \left( \cot\left( \frac{\pi}{2j+2} \right) - 1 \right) .
    \end{align*}
    If $j$ is even, then $\lfloor\frac{j}{2}\rfloor = \frac{j}{2}$ as well as $\cos\left( \frac{\lfloor\frac{j}{2}\rfloor+1}{j+1}\pi\right) = - \cos\left( \frac{\lfloor\frac{j}{2}\rfloor}{j+1}\pi\right)$, which gives
    \begin{align*}
        \Psi(W_{l, j}(x, k)) &= 2\sqrt{k-1} \left( \frac{2\cos\left( \frac{j}{j+1} \cdot \frac{\pi}{2}\right)}{1 - \cos\left(\frac{1}{j+1}\pi\right)}-1 \right)\\
        &= 2\sqrt{k-1} \left( \frac{2\sin\left( \frac{1}{j+1} \cdot \frac{\pi}{2}\right)}{2 \sin^2\left(\frac{\pi}{2j+2}\right)}-1 \right)\\
        &= 2\sqrt{k-1} \left( \csc\left( \frac{\pi}{2j+2} \right)-1 \right) .
    \end{align*}
\end{proof}

\begin{lemma}\label{last_wj_approx}
    Suppose $k \ge 3$. For an even $l \ge 2$, the value $\Psi(W_{l, l+1}(x, k))$ can be approximated via
    $$2 \sqrt{k-1} \left( \cot\left( \frac{\pi}{2l+4}\right) - 2.2 \right) < \Psi(W_{l, l+1}(x, k)) < 2 \sqrt{k-1} \ \cot \left( \frac{\pi}{2l+4}\right) ,$$
    while for an odd $l \ge 1$, the following approximation can be made:
    $$2 \sqrt{k-1} \left( \csc\left( \frac{\pi}{2l+4}\right) - 2.2 \right) < \Psi(W_{l, l+1}(x, k)) < 2 \sqrt{k-1} \ \csc\left( \frac{\pi}{2l+4}\right) .$$
\end{lemma}
\begin{proof}
    Lemma \ref{geronimus_approx} directly gives us
    $$4 \sqrt{k-1} \sum_{h=1}^{\lfloor \frac{l+1}{2} \rfloor} \cos \left( \dfrac{h+0.5}{l+2} \pi \right) < \Psi(W_{l, l+1}(x, k)) < 4 \sqrt{k-1} \sum_{h=1}^{\lfloor \frac{l+1}{2} \rfloor} \cos \left( \dfrac{h-0.5}{l+2} \pi \right) .$$
    We will denote $\zeta = e^{\frac{i \pi}{2l+4}}$. From
    \begin{align*}
        \cos \left( \dfrac{h-0.5}{l+2} \pi \right) &= \dfrac{\zeta^{2h-1}+\zeta^{-(2h-1)}}{2} \, ,\\
        \cos \left( \dfrac{h+0.5}{l+2} \pi \right) &= \dfrac{\zeta^{2h+1}+\zeta^{-(2h+1)}}{2} \, ,
    \end{align*}
    we obtain
    \begin{align*}
        \Psi(W_{l, l+1}(x, k)) &> 2 \sqrt{k-1} \sum_{h=1}^{\lfloor \frac{l+1}{2} \rfloor} \left( \zeta^{2h+1} + \zeta^{-(2h+1)} \right) ,\\
        \Psi(W_{l, l+1}(x, k)) &< 2 \sqrt{k-1} \sum_{h=1}^{\lfloor \frac{l+1}{2} \rfloor} \left( \zeta^{2h-1} + \zeta^{-(2h-1)} \right) .
    \end{align*}
    Since $\zeta \neq 1$, we can apply the standard formula for summing a geometric progression in order to get
    \begin{align*}
        \Psi(W_{l, l+1}(x, k)) &< 2 \sqrt{k-1} \ \zeta^{1-2\lfloor \frac{l+1}{2} \rfloor} \left( 1 + \zeta^2 + \zeta^4 + \cdots + \zeta^{4\lfloor \frac{l+1}{2} \rfloor-2} \right)\\
        &= 2 \sqrt{k-1} \ \zeta^{1-2\lfloor \frac{l+1}{2} \rfloor} \ \dfrac{1 - \zeta^{4\lfloor \frac{l+1}{2} \rfloor}}{1 - \zeta^2}\\
        &= 2 \sqrt{k-1} \ \dfrac{\zeta^{1-2\lfloor \frac{l+1}{2} \rfloor} - \zeta^{2\lfloor \frac{l+1}{2} \rfloor+1}}{1 - \zeta^2}\\
        &= 2 \sqrt{k-1} \ \dfrac{(\zeta^{1-2\lfloor \frac{l+1}{2} \rfloor} - \zeta^{2\lfloor \frac{l+1}{2} \rfloor+1})(1 - \zeta^{-2})}{(1 - \zeta^2)(1 - \zeta^{-2})}\\
        &= 2 \sqrt{k-1} \ \dfrac{(\zeta^{2\lfloor \frac{l+1}{2} \rfloor-1} + \zeta^{1-2\lfloor \frac{l+1}{2} \rfloor}) - (\zeta^{2\lfloor \frac{l+1}{2} \rfloor+1} + \zeta^{-2\lfloor \frac{l+1}{2} \rfloor-1})}{2 - (\zeta^2 + \zeta^{-2})}\\
        &= 2 \sqrt{k-1} \ \dfrac{2 \cos \left( \dfrac{2\lfloor \frac{l+1}{2} \rfloor-1}{2l+4}\pi\right) - 2 \cos \left( \dfrac{2\lfloor \frac{l+1}{2} \rfloor+1}{2l+4}\pi\right)}{2 - 2\cos \left( \dfrac{1}{l+2}\pi\right)}\\
        &= 2 \sqrt{k-1} \ \dfrac{\cos \left( \dfrac{2\lfloor \frac{l+1}{2} \rfloor-1}{2l+4}\pi\right) - \cos \left( \dfrac{2\lfloor \frac{l+1}{2} \rfloor+1}{2l+4}\pi\right)}{1 - \cos \left( \dfrac{1}{l+2}\pi\right)} \, ,
    \end{align*}
    as well as
    \begin{align*}
        \Psi&(W_{l, l+1}(x, k)) >\\
        &> 2 \sqrt{k-1} \left( \zeta^{-1-2\lfloor \frac{l+1}{2} \rfloor} \left( 1 + \zeta^2 + \zeta^4 + \cdots + \zeta^{4\lfloor \frac{l+1}{2} \rfloor+2} \right) - (\zeta + \zeta^{-1}) \right)\\
        &= 2 \sqrt{k-1} \left( \zeta^{-1-2\lfloor \frac{l+1}{2} \rfloor} \ \dfrac{1 - \zeta^{4\lfloor \frac{l+1}{2} \rfloor+4}}{1 - \zeta^2} - (\zeta + \zeta^{-1}) \right)\\
        &= 2 \sqrt{k-1} \left( \dfrac{\zeta^{-1-2\lfloor \frac{l+1}{2} \rfloor} - \zeta^{2\lfloor \frac{l+1}{2} \rfloor+3}}{1-\zeta^2} - (\zeta+\zeta^{-1}) \right)\\
        &= 2 \sqrt{k-1} \left( \dfrac{(\zeta^{-1-2\lfloor \frac{l+1}{2} \rfloor} - \zeta^{2\lfloor \frac{l+1}{2} \rfloor+3})(1-\zeta^{-2})}{(1-\zeta^2)(1-\zeta^{-2})} - (\zeta+\zeta^{-1}) \right)\\
        &= 2 \sqrt{k-1} \left( \dfrac{(\zeta^{2\lfloor \frac{l+1}{2} \rfloor+1}+\zeta^{-2\lfloor \frac{l+1}{2} \rfloor-1}) - (\zeta^{2\lfloor \frac{l+1}{2} \rfloor+3} + \zeta^{-2\lfloor \frac{l+1}{2} \rfloor-3})}{2 - (\zeta^2 + \zeta^{-2})} - (\zeta+\zeta^{-1}) \right)\\
        &= 2 \sqrt{k-1} \left( \dfrac{2\cos \left( \dfrac{2\lfloor \frac{l+1}{2} \rfloor+1}{2l+4}\pi\right) - 2\cos \left( \dfrac{2\lfloor \frac{l+1}{2} \rfloor+3}{2l+4}\pi\right)}{2 - 2\cos \left( \dfrac{1}{l+2}\pi\right)} - 2\cos\left( \dfrac{\pi}{2l+4}\right) \right)\\
        &> 2 \sqrt{k-1} \left( \dfrac{\cos \left( \dfrac{2\lfloor \frac{l+1}{2} \rfloor+1}{2l+4}\pi\right) - \cos \left( \dfrac{2\lfloor \frac{l+1}{2} \rfloor+3}{2l+4}\pi\right)}{1 - \cos \left( \dfrac{1}{l+2}\pi\right)} - 2\cos\left( \dfrac{\pi}{2l+4}\right) \right) \, .
    \end{align*}
    
    \emph{Case }$2 \mid l$:\quad We have $\lfloor \frac{l+1}{2} \rfloor = \frac{l}{2}$, which allows us to conclude
    \begin{align*}
        \Psi(W_{l, l+1}(x, k)) &< 2 \sqrt{k-1} \ \dfrac{\cos \left( \dfrac{l-1}{2l+4}\pi\right) - \cos \left( \dfrac{l+1}{2l+4}\pi\right)}{1 - \cos \left( \dfrac{1}{l+2}\pi\right)}\\
        &= 2 \sqrt{k-1} \ \dfrac{\cos \left( \dfrac{l-1}{l+2} \cdot \dfrac{\pi}{2}\right) - \cos \left( \dfrac{l+1}{l+2} \cdot \dfrac{\pi}{2}\right)}{1 - \cos \left( \dfrac{1}{l+2}\pi\right)}\\
        &= 2 \sqrt{k-1} \ \dfrac{\sin \left( \dfrac{3}{l+2} \cdot \dfrac{\pi}{2}\right) - \sin \left( \dfrac{1}{l+2} \cdot \dfrac{\pi}{2}\right)}{2 \sin^2 \left( \dfrac{\pi}{2l+4}\right)}\\
        &= 2 \sqrt{k-1} \ \dfrac{2 \sin \left( \dfrac{\pi}{2l+4} \right) \cos \left( \dfrac{2\pi}{2l+4} \right)}{2 \sin^2 \left( \dfrac{\pi}{2l+4}\right)}\\
        &= 2 \sqrt{k-1} \ \dfrac{\cos \left( \dfrac{2\pi}{2l+4} \right)}{\sin \left( \dfrac{\pi}{2l+4}\right)}\\
        &< 2 \sqrt{k-1} \ \dfrac{\cos \left( \dfrac{\pi}{2l+4} \right)}{\sin \left( \dfrac{\pi}{2l+4}\right)}\\
        &= 2 \sqrt{k-1} \ \cot \left( \dfrac{\pi}{2l+4} \right)
    \end{align*}
    along with
    \begin{align*}
        \Psi&(W_{l, l+1}(x, k)) >\\
        &> 2 \sqrt{k-1} \left( \dfrac{\cos \left( \dfrac{l+1}{2l+4}\pi\right) - \cos \left( \dfrac{2l+3}{2l+4}\pi\right)}{1 - \cos \left( \dfrac{1}{l+2}\pi\right)} - 2 \cos\left( \dfrac{1}{2l+4}\pi\right) \right)\\
        &= 2 \sqrt{k-1} \left( \dfrac{2\cos \left( \dfrac{l+1}{2l+4}\pi\right)}{2 \sin^2 \left( \dfrac{\pi}{2l+4}\right)} - 2 \cos\left( \dfrac{1}{2l+4}\pi\right) \right)\\
        &= 2 \sqrt{k-1} \left( \dfrac{\sin \left( \dfrac{\pi}{2l+4}\right)}{\sin^2 \left( \dfrac{\pi}{2l+4}\right)} - 2 \cos\left( \dfrac{1}{2l+4}\pi\right) \right)\\
        &> 2 \sqrt{k-1} \left( \csc\left( \dfrac{\pi}{2l+4} \right) - 2.2 \right) .
    \end{align*}
    
    \emph{Case }$2 \nmid l$:\quad We have $\lfloor \frac{l+1}{2} \rfloor = \frac{l+1}{2}$, thus obtaining
    \begin{align*}
        \Psi(W_{l, l+1}(x, k)) &< 2 \sqrt{k-1} \ \dfrac{\cos \left( \dfrac{l}{2l+4}\pi\right) - \cos \left( \dfrac{l+2}{2l+4}\pi\right)}{1 - \cos \left( \dfrac{1}{l+2}\pi\right)}\\
        &= 2 \sqrt{k-1} \ \dfrac{\cos \left( \dfrac{l}{2l+4}\pi\right)}{2 \sin^2 \left( \dfrac{\pi}{2l+4}\right)}\\
        &= 2 \sqrt{k-1} \ \dfrac{\sin \left( \dfrac{2}{2l+4}\pi\right)}{2 \sin^2 \left( \dfrac{\pi}{2l+4}\right)}\\
        &= 2 \sqrt{k-1} \ \dfrac{2 \sin \left( \dfrac{\pi}{2l+4}\right) \cos \left( \dfrac{\pi}{2l+4}\right) }{2 \sin^2 \left( \dfrac{\pi}{2l+4}\right)}\\
        &= 2 \sqrt{k-1} \ \cot \left( \dfrac{\pi}{2l+4} \right)\\
        &< 2 \sqrt{k-1} \ \csc \left( \dfrac{\pi}{2l+4} \right)
    \end{align*}
    as well as
    \begin{align*}
        \Psi&(W_{l, l+1}(x, k)) >\\
        &> 2 \sqrt{k-1} \left( \dfrac{\cos \left( \dfrac{l+2}{2l+4}\pi\right) - \cos \left( \dfrac{l+4}{2l+4}\pi\right)}{1 - \cos \left( \dfrac{1}{l+2}\pi\right)} - 2 \cos\left( \dfrac{1}{2l+4}\pi\right) \right)\\
        &= 2 \sqrt{k-1} \left( \dfrac{ - \cos \left( \dfrac{l+4}{2l+4}\pi\right)}{2 \sin^2 \left( \dfrac{\pi}{2l+4}\right)} - 2 \cos\left( \dfrac{1}{2l+4}\pi\right) \right)\\
        &= 2 \sqrt{k-1} \left( \dfrac{ \sin \left( \dfrac{2}{2l+4}\pi\right)}{2 \sin^2 \left( \dfrac{\pi}{2l+4}\right)} - 2 \cos\left( \dfrac{1}{2l+4}\pi\right) \right)\\
        &= 2 \sqrt{k-1} \left( \dfrac{ 2 \sin \left( \dfrac{\pi}{2l+4}\right) \cos \left( \dfrac{\pi}{2l+4}\right)}{2 \sin^2 \left( \dfrac{\pi}{2l+4}\right)} - 2 \cos\left( \dfrac{1}{2l+4}\pi\right) \right)\\
        &= 2 \sqrt{k-1} \left( \dfrac{\cos \left( \dfrac{\pi}{2l+4}\right)}{\sin \left( \dfrac{\pi}{2l+4}\right)} - 2 \cos\left( \dfrac{1}{2l+4}\pi\right) \right)\\
        &= 2 \sqrt{k-1} \left( \dfrac{1}{\sin \left( \dfrac{\pi}{2l+4}\right)} - \dfrac{1 - \cos \left( \dfrac{\pi}{2l+4}\right)}{\sin \left( \dfrac{\pi}{2l+4}\right)} - 2 \cos\left( \dfrac{1}{2l+4}\pi\right) \right)\\
        &= 2 \sqrt{k-1} \left( \csc\left(\dfrac{\pi}{2l+4}\right) - \dfrac{2 \sin^2 \left( \dfrac{\pi}{4l+8}\right)}{2 \sin \left( \dfrac{\pi}{4l+8}\right)\cos \left( \dfrac{\pi}{4l+8}\right)} - 2 \cos\left( \dfrac{1}{2l+4}\pi\right) \right)\\
        &= 2 \sqrt{k-1} \left( \csc\left(\dfrac{\pi}{2l+4}\right) - \tan\left(\dfrac{\pi}{4l+8}\right) - 2 \cos\left( \dfrac{\pi}{2l+4}\right) \right) .
    \end{align*}
    To complete the proof, it is sufficient to show that
    $$\tan\left(\dfrac{\pi}{4l+8}\right) + 2 \cos\left( \dfrac{\pi}{2l+4}\right) \le 2.2$$
    for all $l \ge 1$. For $l = 1$, we have
    \begin{align*}
        \tan\left( \dfrac{\pi}{12} \right) + 2 \cos\left(\dfrac{\pi}{6}\right) &= (2 - \sqrt{3}) + 2 \cdot \dfrac{\sqrt{3}}{2} = 2 < 2.2 \, .
    \end{align*}
    For $l \ge 2$, we get $\tan\left( \dfrac{\pi}{4l+8} \right) \le \tan\left( \dfrac{\pi}{16} \right) = \sqrt{4 + 2 \sqrt{2}} - \sqrt{2} - 1 < 0.2$. Together with $2 \cos\left(\dfrac{\pi}{2l+4}\right) \le 2$, this gives the desired result.
\end{proof}

With the direct help of Lemma \ref{all_wj_exact} and Lemma \ref{last_wj_approx}, we are able to formulate the following proof.

\bigskip\noindent
{\em Proof of Theorem \ref{energy_approx}}.\quad
First of all, it is easy to prove Eq.\ (\ref{simple_l2}) by taking into consideration Theorem \ref{spectra}. It immediately follows that the energy of $d(l, 2)$ must be equal to the sum
$$E(d(l, 2)) = 2 \sum_{h=1}^{l} 2 \cos \left( \dfrac{h}{2l+2}\pi \right)$$
Thankfully, Lemma \ref{all_wj_exact} has already shown us how this exact sum can be computed in an elegant way. In the remainder of the proof we will suppose that $k \ge 3$ and focus on proving Eq.\ (\ref{complicated_l3_1}) and Eq.\ (\ref{complicated_l3_2}).

From Eq.\ (\ref{dendrimer-formula}), we easily obtain
\begin{align*}
    E(d(l, k)) = \Psi(W_{l, l+1}(x, k)) &+ (k-1)\Psi(W_{l, l}(x, k))\\
    &+ \sum_{j=1}^{l-1}k(k-2)(k-1)^{l-1-j}\Psi(W_{l, j}(x, k)) \, .
\end{align*}
The expression for $E(d(l, k))$ can be rewritten in the following way:
\newpage
\begin{align*}
    E(d(l, k)) = \Psi&(W_{l, l+1}(x, k)) + (k-1)\Psi(W_{l, l}(x, k))\\
    &+ \sum_{j=1}^{l-1}((k-1)^2 - 1) (k-1)^{l-1-j}\Psi(W_{l, j}(x, k))\\
    = \Psi&(W_{l, l+1}(x, k)) + (k-1)\Psi(W_{l, l}(x, k))\\
    &+ \sum_{j=1}^{l-1}(k-1)^{l+1-j}\Psi(W_{l, j}(x, k)) - \sum_{j=1}^{l-1}(k-1)^{l-1-j}\Psi(W_{l, j} (x, k))\\
    = \Psi&(W_{l, l+1}(x, k)) + (k-1)\Psi(W_{l, k}(x, k))\\
    &+ \sum_{j=0}^{l-2}(k-1)^{l-j}\Psi(W_{l, j+1}(x, k)) - \sum_{j=2}^{l}(k-1)^{l-j}\Psi(W_{l, j-1}(x, k)) \, ,
\end{align*}
thus giving
\begin{align*}
    E(d(l, k)) = \sum_{j=0}^{l}(k-1)^{l-j}\Psi(W_{l, j+1}(x, k)) - \sum_{j=2}^{l}(k-1)^{l-j}\Psi(W_{l, j-1}(x, k)) \, .
\end{align*}
We know that $\Psi(W_{l, 1}(x, k)) = 0$ since $W_{l, 1}(x, k) = x$, which means that
\begin{align*}
    E(d(l, k)) = \sum_{j=1}^{l}(k-1)^{l-j}\Psi(W_{l, j+1}(x, k)) &+ (k-1)^l \Psi(W_{l, 1}(x, k))\\
    &- \sum_{j=2}^{l}(k-1)^{l-j}\Psi(W_{l, j-1}(x, k))\\
    = \sum_{j=1}^{l}(k-1)^{l-j}\Psi(W_{l, j+1}(x, k)) &- \sum_{j=2}^{l}(k-1)^{l-j}\Psi(W_{l, j-1}(x, k)) \, .
\end{align*}
Also, $\Psi(W_{l, 0}(x, k)) = 0$, which implies
\begin{align*}
    E(d(l, k)) = \sum_{j=1}^{l}(k-1)^{l-j}\Psi(W_{l, j+1}(x, k)) &- \sum_{j=2}^{l}(k-1)^{l-j}\Psi(W_{l, j-1}(x, k))\\
    &- (k-1)^{l-1} \Psi(W_{l, 0}(x, k))\\
    = \sum_{j=1}^{l}(k-1)^{l-j}\Psi(W_{l, j+1}(x, k)) &- \sum_{j=1}^{l}(k-1)^{l-j}\Psi(W_{l, j-1}(x, k)) \, ,\\
\end{align*}
from which we get
\begin{align*}
    E(d(l, k)) &= \sum_{j=1}^{l}(k-1)^{l-j}[\Psi(W_{l, j+1}(x, k)) - \Psi(W_{l, j-1}(x, k))] \, .
\end{align*}
The implementation of Lemma \ref{all_wj_exact} gives
\begin{align*}
    \Psi&(W_{l, j+1}(x, k)) - \Psi(W_{l, j-1}(x, k)) = \\
    &= \begin{cases}
        2 \sqrt{k-1} \left( \cot\left( \dfrac{\pi}{2j+4}\right) - 1\right) - 2 \sqrt{k-1} \left( \cot\left( \dfrac{\pi}{2j}\right) - 1 \right), & 2 \mid j\\[12pt]
        2 \sqrt{k-1} \left( \csc\left( \dfrac{\pi}{2j+4}\right) - 1\right) - 2 \sqrt{k-1} \left( \csc\left( \dfrac{\pi}{2j}\right) - 1 \right), & 2 \nmid j
    \end{cases}\\
    &= \begin{cases}
        (k-1)^{1/2} \left( 2 \cot\left( \dfrac{\pi}{2j+4}\right) - 2  \cot\left( \dfrac{\pi}{2j}\right) \right), & 2 \mid j\\[12pt]
        (k-1)^{1/2} \left( 2 \csc\left( \dfrac{\pi}{2j+4}\right) - 2  \csc\left( \dfrac{\pi}{2j}\right) \right), & 2 \nmid j
    \end{cases}\\
    &=  f_{j-1} (k-1)^{1/2} \, ,
\end{align*}
for each $1 \le j \le l-1$, which further implies
\begin{align*}
    E(d(l, k)) &= \sum_{j=1}^{l-1} f_{j-1} (k-1)^{l+1/2-j} + [\Psi(W_{l, l+1}(x, k)) - \Psi(W_{l, l-1}(x, k))]\\
    &= \sum_{j=0}^{l-2} f_j (k-1)^{l-1/2-j} + [\Psi(W_{l, l+1}(x, k)) - \Psi(W_{l, l-1}(x, k))] \, .
\end{align*}
Lemma \ref{last_wj_approx} helps us to approximate the final term $\Psi(W_{l, l+1}(x, k)) - \Psi(W_{l, l-1}(x, k))$. If $l$ is odd, then
\begin{align*}
    \Psi&(W_{l, l+1}(x, k)) - \Psi(W_{l, l-1}(x, k)) <\\
    &< 2 \sqrt{k-1} \ \csc\left( \frac{\pi}{2l+4}\right) - 2 \sqrt{k-1} \left( \csc\left( \dfrac{\pi}{2l} \right) - 1 \right)\\
    &= 2 \sqrt{k-1} \left( \csc\left( \frac{\pi}{2l+4}\right) - \csc\left( \dfrac{\pi}{2l} \right) \right) + 2 \sqrt{k-1}\\
    &= f_{l-1} (k-1)^{1/2} + 2 (k-1)^{1/2}
\end{align*}
and
\begin{align*}
    \Psi&(W_{l, l+1}(x, k)) - \Psi(W_{l, l-1}(x, k)) >\\
    &> 2 \sqrt{k-1} \left( \csc\left( \frac{\pi}{2l+4}\right) - 2.2 \right) - 2 \sqrt{k-1} \left( \csc\left( \dfrac{\pi}{2l} \right) - 1 \right)\\
    &= 2 \sqrt{k-1} \left( \csc\left( \frac{\pi}{2l+4}\right) - \csc\left( \dfrac{\pi}{2l} \right) \right) - 2.4 \sqrt{k-1}\\
    &= f_{l-1} (k-1)^{1/2} - 2.4 (k-1)^{1/2} \, .
\end{align*}
If $l$ is even, then we similarly obtain
\begin{align*}
    \Psi&(W_{l, l+1}(x, k)) - \Psi(W_{l, l-1}(x, k)) <\\
    &< 2 \sqrt{k-1} \ \cot\left( \frac{\pi}{2l+4}\right) - 2 \sqrt{k-1} \left( \cot\left( \dfrac{\pi}{2l} \right) - 1 \right)\\
    &= 2 \sqrt{k-1} \left( \cot\left( \frac{\pi}{2l+4}\right) - \cot\left( \dfrac{\pi}{2l} \right) \right) + 2 \sqrt{k-1}\\
    &= f_{l-1} (k-1)^{1/2} + 2 (k-1)^{1/2}
\end{align*}
and
\begin{align*}
    \Psi&(W_{l, l+1}(x, k)) - \Psi(W_{l, l-1}(x, k)) >\\
    &> 2 \sqrt{k-1} \left( \cot\left( \frac{\pi}{2l+4}\right) - 2.2 \right) - 2 \sqrt{k-1} \left( \cot\left( \dfrac{\pi}{2l} \right) - 1 \right)\\
    &= 2 \sqrt{k-1} \left( \cot\left( \frac{\pi}{2l+4}\right) - \cot\left( \dfrac{\pi}{2l} \right) \right) - 2.4 \sqrt{k-1}\\
    &= f_{l-1} (k-1)^{1/2} - 2.4 (k-1)^{1/2} \, .
\end{align*}
Taking everything into consideration, we get the approximations
\begin{align*}
    E(d(l, k)) &< \sum_{j=0}^{l-2} f_j (k-1)^{l-1/2-j} + f_{l-1} (k-1)^{1/2} + 2 (k-1)^{1/2}\\
    &= \sum_{j=0}^{l-1} f_j (k-1)^{l-1/2-j} + 2 (k-1)^{1/2}
\end{align*}
and
\begin{align*}
    E(d(l, k)) &> \sum_{j=0}^{l-2} f_j (k-1)^{l-1/2-j} + f_{l-1} (k-1)^{1/2} - 2.4 (k-1)^{1/2}\\
    &= \sum_{j=0}^{l-1} f_j (k-1)^{l-1/2-j} - 2.4 (k-1)^{1/2}
\end{align*}
which complete the proof. \qed

Theorem \ref{energy_approx} provides a way to approximate the energy of a dendrimer $d(l, k)$ via two expressions that resemble polynomials. To be more precise, these expressions represent a linear combination of the power terms $(k-1)^{h-1/2}, \ h \in \mathbb{N}$, where the corresponding coefficients depend solely on $l$, not $k$. This makes it easier to analyze the asymptotic properties of $E(d(l, k))$, leading us to the first of our two main theorems:
\begin{theoremx}\label{asymptotic_behavior}
    For any fixed value of $l \ge 1$, we have
    \begin{equation}\label{k_to_inf}
        E(d(l, k)) \sim 2 (k-1)^{l-1/2} \qquad \mbox{as $\ k \to \infty$} \, .
    \end{equation}
    Also, for any fixed value of $k \ge 3$, we have
    \begin{equation}\label{l_to_inf}
        E(d(l, k)) \sim \mu_k (k-1)^{l-1/2} \qquad \mbox{as $\ l \to \infty$} \, .
    \end{equation}
    where $\mu_k$ is the positive real number which represents the sum of the convergent positive series
    $$\sum_{j = 0}^{\infty} f_j (k-1)^{-j}$$
    defined by
    $$f_j = \begin{cases}
        2 \csc\left( \dfrac{\pi}{2j+6} \right) - 2\csc\left( \dfrac{\pi}{2j+2} \right), & 2 \mid j \, ,\\
        2 \cot\left( \dfrac{\pi}{2j+6} \right) - 2\cot\left( \dfrac{\pi}{2j+2} \right), & 2 \nmid j \, .
    \end{cases}$$
\end{theoremx}

In order to make the proof of Theorem \ref{asymptotic_behavior} easier to follow, we are going to need the following auxiliary lemma:
\begin{lemma}\label{cool_sequences}
    If $(a_j)_{j \in \mathbb{N}_0}$ and $(b_j)_{j \in \mathbb{N}_0}$ are real sequences such that
    \begin{align*}
        a_j &= 2 \csc \left( \dfrac{\pi}{2j+6} \right) - 2 \csc \left( \dfrac{\pi}{2j+2} \right) ,\\
        b_j &= 2 \cot \left( \dfrac{\pi}{2j+6} \right) - 2 \cot \left( \dfrac{\pi}{2j+2} \right) ,
    \end{align*}
    for each $j \in \mathbb{N}_0$, then both of these sequences converge to $\dfrac{8}{\pi}$. Also, $(a_j)_{j \in \mathbb{N}_0}$ is strictly increasing, while $(b_j)_{j \in \mathbb{N}_0}$ is strictly decreasing.
\end{lemma}
\begin{proof}
    We have
    \begin{align*}
        a_j &= 2 \csc \left( \frac{\pi}{2j+6} \right) - 2 \csc \left( \frac{\pi}{2j+2} \right)\\
        &= 2 \ \frac{ \sin \left( \frac{\pi}{2j+2} \right) - \sin \left( \frac{\pi}{2j+6} \right)}{\sin \left( \frac{\pi}{2j+2} \right) \sin \left( \frac{\pi}{2j+6} \right)}\\
        &= 2 \ \frac{ 2 \sin \left( \frac{\frac{\pi}{2j+2} - \frac{\pi}{2j+6}}{2}\right) \cos \left( \frac{\frac{\pi}{2j+2} + \frac{\pi}{2j+6}}{2}\right)}{\sin \left( \frac{\pi}{2j+2} \right) \sin \left( \frac{\pi}{2j+6} \right)}\\
        &= 4 \ \frac{ \sin \left( \frac{\pi}{2(j+1)(j+3)} \right) \cos \left( \frac{(j+2)\pi}{2(j+1)(j+3)} \right)}{\sin \left( \frac{\pi}{2j+2} \right) \sin \left( \frac{\pi}{2j+6} \right)}\\
        &= 4 \ \frac{2}{\pi} \cdot \frac{\frac{\pi}{2j+2}}{\sin \left( \frac{\pi}{2j+2} \right)} \cdot \frac{\frac{\pi}{2j+6}}{\sin \left( \frac{\pi}{2j+6} \right)} \cdot 
        \frac{\sin \left( \frac{\pi}{2(j+1)(j+3)} \right)}{\frac{\pi}{2(j+1)(j+3)}} \cdot \cos \left( \frac{(j+2)\pi}{2(j+1)(j+3)} \right)\\
        &= \frac{8}{\pi} \cdot \frac{\frac{\pi}{2j+2}}{\sin \left( \frac{\pi}{2j+2} \right)} \cdot \frac{\frac{\pi}{2j+6}}{\sin \left( \frac{\pi}{2j+6} \right)} \cdot 
        \frac{\sin \left( \frac{\pi}{2(j+1)(j+3)} \right)}{\frac{\pi}{2(j+1)(j+3)}} \cdot \cos \left( \frac{(j+2)\pi}{2(j+1)(j+3)} \right) .
    \end{align*}
    By using the well known properties of limits together with the famous formula $\lim\limits_{x \to 0}\dfrac{\sin x}{x} = 1$, it becomes obvious that $\lim\limits_{j \to \infty}a_j = \dfrac{8}{\pi}$. Similarly, we obtain
    \begin{align*}
        b_j &= 2 \cot \left( \frac{\pi}{2j+6} \right) - 2 \cot \left( \frac{\pi}{2j+2} \right)\\
        &= 2 \ \frac{ \sin \left( \frac{\pi}{2j+2} \right) \cos \left( \frac{\pi}{2j+6} \right) - \sin \left( \frac{\pi}{2j+6} \right) \cos \left( \frac{\pi}{2j+2} \right)}{\sin \left( \frac{\pi}{2j+2} \right) \sin \left( \frac{\pi}{2j+6} \right)}\\
        &= 2 \ \frac{ \sin \left( \frac{\pi}{2j+2} - \frac{\pi}{2j+6} \right)}{\sin \left( \frac{\pi}{2j+2} \right) \sin \left( \frac{\pi}{2j+6} \right)}\\
        &= 2 \ \frac{ \sin \left( \frac{\pi}{(j+1)(j+3)} \right)}{\sin \left( \frac{\pi}{2j+2} \right) \sin \left( \frac{\pi}{2j+6} \right)}\\
        &= 2 \ \frac{4}{\pi} \cdot \frac{\frac{\pi}{2j+2}}{\sin \left( \frac{\pi}{2j+2} \right)} \cdot \frac{\frac{\pi}{2j+6}}{\sin \left( \frac{\pi}{2j+6} \right)} \cdot 
        \frac{\sin \left( \frac{\pi}{(j+1)(j+3)} \right)}{\frac{\pi}{(j+1)(j+3)}}\\
        &= \frac{8}{\pi} \cdot \frac{\frac{\pi}{2j+2}}{\sin \left( \frac{\pi}{2j+2} \right)} \cdot \frac{\frac{\pi}{2j+6}}{\sin \left( \frac{\pi}{2j+6} \right)} \cdot 
        \frac{\sin \left( \frac{\pi}{(j+1)(j+3)} \right)}{\frac{\pi}{(j+1)(j+3)}} \, ,
    \end{align*}
    which implies $\lim\limits_{j \to \infty}b_j = \dfrac{8}{\pi}$.
    
    Let us define the function $y_1(x) = \csc \left( \dfrac{\pi}{x} \right)$ on $[2, +\infty)$. This function is obviously twice differentiable everywhere. We know that
    \begin{align*}
        a_j - a_{j-1} &= 2 \csc \left( \dfrac{\pi}{2j+6} \right) - 2 \csc \left( \dfrac{\pi}{2j+2} \right) - 2 \csc \left( \dfrac{\pi}{2j+4} \right) + 2 \csc \left( \dfrac{\pi}{2j} \right)\\
        &= 2(y_1(2j+6) - y_1(2j+4)) - 2(y_1(2j+2)- y_1(2j))
    \end{align*}
    for each $j \ge 1$. Since the function $y_1$ is differentiable everywhere, we can use the mean value theorem to get
    \begin{align*}
        a_j - a_{j-1} &= 2((2j+6)-(2j+4))y_1^{\prime}(\xi_1) - 2((2j+2)-2j)y_1^{\prime}(\xi_2)\\
        &= 4 y_1^{\prime}(\xi_1) - 4 y_1^{\prime}(\xi_2)
    \end{align*}
    for some $\xi_1 \in (2j+4, 2j+6),\ \xi_2 \in (2j, 2j+2)$. Since $\xi_1 > \xi_2 > 2$, in order to prove that $(a_j)_{j \in \mathbb{N}_0}$ is strictly increasing, it is sufficient to show that $y_1^{\prime\prime}(x)$ is positive on $(2, +\infty)$. We compute
    \begin{align*}
        y_1^{\prime}(x) &= -\dfrac{\cos\left( \dfrac{\pi}{x} \right)}{\sin^2 \left( \dfrac{\pi}{x}\right)} \cdot \left( - \dfrac{\pi}{x^2} \right) = \dfrac{\cos\left( \dfrac{\pi}{x} \right)}{\sin^2 \left( \dfrac{\pi}{x}\right)} \cdot \dfrac{\pi}{x^2} = \dfrac{\pi}{x^2} \csc \left( \dfrac{\pi}{x}\right) \cot \left( \dfrac{\pi}{x}\right) ,
    \end{align*}
    which further gives
    \begin{align*}
        y_1^{\prime\prime}(x) &= \left( \dfrac{\pi}{x^2} \right)^{\prime} \csc \left( \dfrac{\pi}{x}\right) \cot \left( \dfrac{\pi}{x}\right) + \dfrac{\pi}{x^2} \csc^{\prime} \left( \dfrac{\pi}{x}\right) \cot \left( \dfrac{\pi}{x}\right) + \dfrac{\pi}{x^2} \csc \left( \dfrac{\pi}{x}\right) \cot^{\prime} \left( \dfrac{\pi}{x}\right)\\
        &= - \dfrac{2\pi}{x^3} \csc \left( \dfrac{\pi}{x}\right) \cot \left( \dfrac{\pi}{x}\right) + \dfrac{\pi}{x^2} \left( \dfrac{\pi}{x^2} \csc \left( \dfrac{\pi}{x}\right) \cot \left( \dfrac{\pi}{x}\right) \right) \cot \left( \dfrac{\pi}{x}\right)\\
        & \qquad\qquad + \dfrac{\pi}{x^2} \csc \left( \dfrac{\pi}{x}\right) \left( -\csc^2 \left( \dfrac{\pi}{x} \right) \right) \cdot \left( - \dfrac{\pi}{x^2} \right)\\
        &= - \dfrac{2\pi}{x^3} \csc \left( \dfrac{\pi}{x}\right) \cot \left( \dfrac{\pi}{x}\right) + \dfrac{\pi^2}{x^4} \left( \csc \left( \dfrac{\pi}{x} \right) \cot^2 \left( \dfrac{\pi}{x} \right) + \csc^3 \left( \dfrac{\pi}{x} \right) \right)\\
        &= \dfrac{\pi}{x^4} \csc \left( \dfrac{\pi}{x} \right) \left( -2x \cot \left( \dfrac{\pi}{x} \right) + \pi \cot^2 \left( \dfrac{\pi}{x} \right) + \pi \csc^2 \left( \dfrac{\pi}{x} \right)\right) .
    \end{align*}
    Here, it is possible to conclude that for $x \in (2, +\infty)$
    \begin{align*}
        y_1^{\prime\prime}(x) > 0 &\iff -2x \cot \left( \dfrac{\pi}{x} \right) + \pi \cot^2 \left( \dfrac{\pi}{x} \right) + \pi \csc^2 \left( \dfrac{\pi}{x} \right) > 0\\
        &\iff \cot^2 \left( \dfrac{\pi}{x} \right) + \csc^2 \left( \dfrac{\pi}{x} \right) > \dfrac{2x}{\pi} \cot \left( \dfrac{\pi}{x} \right)\\
        &\iff \cos^2 \left( \dfrac{\pi}{x} \right) + 1 > \dfrac{2x}{\pi} \sin \left( \dfrac{\pi}{x} \right) \cos \left( \dfrac{\pi}{x} \right)\\
        &\iff \cos^2 \left( \dfrac{\pi}{x} \right) + 1 > \dfrac{2 \sin \left( \dfrac{\pi}{x} \right) \cos \left( \dfrac{\pi}{x} \right) }{\dfrac{\pi}{x}} .
    \end{align*}
    Hence, in order to prove that $y_1^{\prime\prime}(x) > 0$ for all $x \in (2, +\infty)$, it is sufficient to show that $\cos^2 (\theta) + 1 > \dfrac{2\sin\theta\cos\theta}{\theta}$ for each $\theta \in (0, \pi/2)$. However, we know that
    \begin{alignat*}{2}
        && \cos^2 \theta + 1 &> \dfrac{2\sin\theta\cos\theta}{\theta}\\
        \iff && \cos^2 \theta - 2\cos\theta + 1 &> \dfrac{2\sin\theta\cos\theta}{\theta} - 2\cos\theta\\
        \iff && (\cos\theta-1)^2 &> 2\cos\theta \cdot \left( \dfrac{\sin\theta}{\theta} - 1 \right) ,
    \end{alignat*}
    which obviously must hold for all $\theta \in (0, \pi/2)$, since $(\cos\theta-1)^2 > 0$, while $\dfrac{\sin\theta}{\theta} - 1 < 0$ due to $0 < \sin\theta < \theta$, which is known to hold on $(0, \pi/2)$.
    
    Now we define the function $y_2(x) = \cot\left( \dfrac{\pi}{x} \right)$ on $[2, +\infty)$. This function is also twice differentiable everywhere. We have
    \begin{align*}
        b_j - b_{j-1} &= 2 \cot \left( \dfrac{\pi}{2j+6} \right) - 2 \cot \left( \dfrac{\pi}{2j+2} \right) - 2 \cot \left( \dfrac{\pi}{2j+4} \right) + 2 \cot \left( \dfrac{\pi}{2j} \right)\\
        &= 2(y_2(2j+6) - y_2(2j+4)) - 2(y_2(2j+2)- y_2(2j))
    \end{align*}
    for each $j \ge 1$. Due to the differentiability of $y_2$ on its entire domain, we can implement the mean value theorem to obtain
    \begin{align*}
        b_j - b_{j-1} &= 2((2j+6)-(2j+4))y_2^{\prime}(\eta_1) - 2((2j+2)-2j)y_2^{\prime}(\eta_2)\\
        &= 4 y_2^{\prime}(\eta_1) - 4 y_2^{\prime}(\eta_2)
    \end{align*}
    for some $\eta_1 \in (2j+4, 2j+6),\ \eta_2 \in (2j, 2j+2)$. Because of $\eta_1 > \eta_2 > 2$, in order to prove that $(b_j)_{j \in \mathbb{N}_0}$ is strictly decreasing, it is sufficient to show that $y_2^{\prime\prime}(x)$ is negative on $(2, +\infty)$. We compute
    \begin{align*}
        y_2^{\prime}(x) &= -\csc^2\left( \dfrac{\pi}{x} \right) \cdot \left( - \dfrac{\pi}{x^2} \right) = \dfrac{\pi}{x^2} \csc^2\left( \dfrac{\pi}{x} \right)
    \end{align*}
    along with
    \begin{align*}
        y_2^{\prime\prime}(x) &= \left( \dfrac{\pi}{x^2} \right)^{\prime} \csc^2 \left( \dfrac{\pi}{x}\right) + 2 \ \dfrac{\pi}{x^2} \csc \left( \dfrac{\pi}{x}\right) \csc^{\prime} \left( \dfrac{\pi}{x}\right) \\
        &= - \dfrac{2\pi}{x^3} \csc^2 \left( \dfrac{\pi}{x}\right) + 2 \ \dfrac{\pi}{x^2} \csc \left( \dfrac{\pi}{x}\right) \cdot \left( \dfrac{\pi}{x^2} \csc \left( \dfrac{\pi}{x}\right) \cot \left( \dfrac{\pi}{x}\right) \right)\\
        &= \dfrac{\pi}{x^4} \csc^2 \left( \dfrac{\pi}{x}\right) \left( -2x + 2\pi \cot \left( \dfrac{\pi}{x}\right) \right) .
    \end{align*}
    Thus, for $x \in (2, +\infty)$ we have
    \begin{align*}
        y_2^{\prime\prime}(x) < 0 &\iff -2x + 2\pi \cot \left( \dfrac{\pi}{x}\right) < 0\\
        &\iff \dfrac{\pi}{x} \cot \left( \dfrac{\pi}{x}\right) < 1\\
        &\iff \dfrac{\pi}{x} < \tan \left( \dfrac{\pi}{x}\right)
    \end{align*}
    However, it is known that the inequality $\tan \theta > \theta$ holds on $(0, \pi/2)$, which proves that the sequence $(b_j)_{j \in \mathbb{N}_0}$ must be strictly decreasing.
\end{proof}

We are now in the position to prove Theorem \ref{asymptotic_behavior} by extensively relying on the sequences $(a_j)_{j \in \mathbb{N}_0}$ and $(b_j)_{j \in \mathbb{N}_0}$ defined in Lemma \ref{cool_sequences}, as well their properties which we have shown.

\bigskip\noindent
{\em Proof of Theorem \ref{asymptotic_behavior}}.\quad
First of all, we are going to prove Eq.\ (\ref{k_to_inf}). Directly from Theorem \ref{energy_approx}, we obtain:
\begin{equation}\label{quotient_approx_1}
    \dfrac{E(d(l, k))}{(k-1)^{l-1/2}} < \sum_{j = 0}^{l-1} f_j (k-1)^{-j} + 2(k-1)^{-(l-1)} \, ,
\end{equation}
as well as
\begin{equation}\label{quotient_approx_2}
    \dfrac{E(d(l, k))}{(k-1)^{l-1/2}} > \sum_{j = 0}^{l-1} f_j (k-1)^{-j} - 2.4(k-1)^{-(l-1)} \, .
\end{equation}
Provided $l \ge 2$, it is clear that when the $l$ variable is fixed and $k \to \infty$, then both the right-hand side in Eq.\ (\ref{quotient_approx_1}) and the right-hand side in Eq.\ (\ref{quotient_approx_2}) tend to $f_0$. By the squeeze theorem, we get that
\begin{align*}
    \lim_{k \to \infty} \dfrac{E(d(l, k))}{(k-1)^{l-1/2}} = f_0 \, .
\end{align*}
Taking into consideration that $f_0 = a_0 = 2$, Eq.\ (\ref{k_to_inf}) is proven for $l \ge 2$. For $l = 1$, the formula needs to be proven directly. From Eq.\ (\ref{l_is_1}) we see that $E(d(1, k)) = 2 \sqrt{k}$. Having this in mind, it is obvious that
\begin{align*}
    \lim_{k \to \infty} \dfrac{E(d(1, k))}{(k-1)^{1/2}} = 2 \, ,
\end{align*}
which completes the proof of Eq.\ (\ref{k_to_inf}).

The next step is to prove that the positive series
$$\sum_{j = 0}^{\infty} f_j (k-1)^{-j}$$
is convergent. From Lemma \ref{cool_sequences}, it is clear that all the $a_j$ must be smaller than $\dfrac{8}{\pi}$, while all the $b_j$ must be greater than $\dfrac{8}{\pi}$. Given the fact that $f_j = a_j$ when $j$ is even and $f_j = b_j$ when $j$ is odd, it is easy to establish that
$$ \sup_{j \in \mathbb{N}_0} f_j = f_1$$
In other words, if we denote $F_j = \sum\limits_{h=0}^{j} f_h (k-1)^{-h}$, then
$$ F_j \le \sum_{h=0}^{j} f_1 (k-1)^{-h}$$
for each $j \in \mathbb{N}_0$. Subsequently,
\begin{align*}
    F_j &\le f_1 \sum_{h=0}^{j} (k-1)^{-h}\\
    &= f_1 \ \dfrac{1 - (k-1)^{-j-1}}{1-(k-1)^{-1}}\\
    &< \dfrac{f_1}{1-(k-1)^{-1}}
\end{align*}
which means that the positive series $\sum\limits_{j = 0}^{\infty} f_j (k-1)^{-j}$ must have a sequence of partial sums which is bounded. This implies that the series is convergent.

Finally, we are going to prove Eq.\ (\ref{l_to_inf}). Suppose that the variable $k \ge 3$ is fixed. From Eq.\ (\ref{quotient_approx_1}) we get
\begin{align*}
    \limsup_{l \to \infty} \dfrac{E(d(l, k))}{(k-1)^{l-1/2}} \le \limsup_{l \to \infty} \left( \sum_{j = 0}^{l-1} f_j (k-1)^{-j} + 2(k-1)^{-(l-1)} \right) .
\end{align*}
Having proven the convergence of $\sum\limits_{j = 0}^{\infty} f_j (k-1)^{-j}$, we know that its sum is a positive real number $\mu_k$. From
\begin{align*}
    \lim\limits_{l \to \infty} \sum_{j = 0}^{l-1} f_j (k-1)^{-j} &= \mu_k \, ,\\
    \lim\limits_{l \to \infty} 2(k-1)^{-(l-1)} &= 0 \, ,
\end{align*}
we obtain
$$\limsup_{l \to \infty} \dfrac{E(d(l, k))}{(k-1)^{l-1/2}} \le \mu_k \, .$$
Similarly, from Eq.\ (\ref{quotient_approx_2}) we have
\begin{align*}
    \liminf_{l \to \infty} \dfrac{E(d(l, k))}{(k-1)^{l-1/2}} \ge \liminf_{l \to \infty} \left( \sum_{j = 0}^{l-1} f_j (k-1)^{-j} - 2.4(k-1)^{-(l-1)} \right) ,
\end{align*}
which ultimately gives
$$\liminf_{l \to \infty} \dfrac{E(d(l, k))}{(k-1)^{l-1/2}} \ge \mu_k \, .$$
From
$$ \mu_k \le \liminf_{l \to \infty} \dfrac{E(d(l, k))}{(k-1)^{l-1/2}} \le \limsup_{l \to \infty} \dfrac{E(d(l, k))}{(k-1)^{l-1/2}} \le \mu_k$$
we get $\lim\limits_{l \to \infty} \dfrac{E(d(l, k))}{(k-1)^{l-1/2}} = \mu_k$, which proves Eq.\ (\ref{l_to_inf}), as desired. \qed

We will end this paper by proving our second main result, along with one of its direct corollaries.
\begin{theoremx}\label{best_approx}
    For a given dendrimer $d(l, k)$, where $k \ge 3$ and $l \ge 2$, we have
    \begin{align}\label{final_theorem_1}
        E(d(l,k)) &< (k-1)^{l-1/2} \left( 2 + \dfrac{0.5 + \sqrt{2} + \sqrt{3} + \sqrt{5}}{k-1} \right) ,\\
        \label{final_theorem_2}E(d(l,k)) &> (k-1)^{l-1/2} \left( 2 + \dfrac{2 \sqrt{2}}{k-1} \right) .
    \end{align}
\end{theoremx}
\begin{proof}
    For $l = 2$, we are going to prove that both Eq.\ (\ref{final_theorem_1}) and Eq.\ (\ref{final_theorem_2}) hold directly. From Eq.\ (\ref{l_is_2}), we conclude that $E(d(2, k)) = 2(k-1)^{3/2} + 2 \sqrt{2k-1}$, which gives us
    \begin{align*}
        \dfrac{E(d(2, k))}{(k-1)^{3/2}} &= 2 + 2 (k-1)^{-3/2} \sqrt{2k-1}\\
        &= 2 + \dfrac{2}{k-1} \sqrt{\dfrac{2k-1}{k-1}}\\
        &= 2 + \dfrac{2}{k-1} \sqrt{2 + \dfrac{1}{k-1}}
    \end{align*}
    We obviously have $\sqrt{2 + \dfrac{1}{k-1}} > \sqrt{2}$, which proves Eq.\ (\ref{final_theorem_2}). However, due to $k \ge 3$, we also have $\sqrt{2 + \dfrac{1}{k-1}} \le \sqrt{\dfrac{5}{2}} < \sqrt{5} < 0.5 + \sqrt{2} + \sqrt{3} + \sqrt{5}$, thereby proving Eq.\ (\ref{final_theorem_1}) as well.
    
    For $l = 3$, the expressions will again be proven directly. From Theorem \ref{energy_approx}, we get
    \begin{align*}
        E(d(3, k)) &< \sum_{j = 0}^{2} f_j (k-1)^{5/2-j} + 2(k-1)^{1/2} \, ,\\
        E(d(3, k)) &> \sum_{j = 0}^{2} f_j (k-1)^{5/2-j} - 2.4(k-1)^{1/2} \, ,
    \end{align*}
    where
    \begin{align*}
        f_0 &= 2\csc\left( \dfrac{\pi}{6} \right) - 2\csc\left( \dfrac{\pi}{2} \right) = 2 \, ,\\
        f_1 &= 2\cot\left( \dfrac{\pi}{8} \right) - 2\cot\left( \dfrac{\pi}{4} \right) = 2 \sqrt{2} \, ,\\
        f_2 &= 2\csc\left( \dfrac{\pi}{10} \right) - 2\csc\left( \dfrac{\pi}{6} \right) = 2 \sqrt{5} - 2 \, .
    \end{align*}
    Hence, we obtain
    \begin{align*}
        \dfrac{E(d(3, k))}{(k-1)^{5/2}} &< 2 + \dfrac{2 \sqrt{2}}{k-1} + \dfrac{2 \sqrt{5} - 2 + 2}{(k-1)^2} \, ,\\
        \dfrac{E(d(3, k))}{(k-1)^{5/2}} &> 2 + \dfrac{2 \sqrt{2}}{k-1} + \dfrac{2 \sqrt{5} - 2 - 2.4}{(k-1)^2} \, .
    \end{align*}
    We know that $2\sqrt{5} > 4.4 = 2 + 2.4$, which means that Eq.\ (\ref{final_theorem_2}) must hold for $l = 3$. On the other hand, we get
    \begin{align*}
        \dfrac{2 \sqrt{2}}{k-1} + \dfrac{2 \sqrt{5} - 2 + 2}{(k-1)^2} &= \dfrac{2 \sqrt{2}}{k-1} + \dfrac{2 \sqrt{5}}{(k-1)^2}\\
        &= \dfrac{2 \sqrt{2}}{k-1} + \dfrac{1}{k-1} \cdot \dfrac{2 \sqrt{5}}{k-1}\\
        &\le \dfrac{2 \sqrt{2}}{k-1} + \dfrac{1}{2} \cdot \dfrac{2 \sqrt{5}}{k-1}\\
        &= \dfrac{2 \sqrt{2} + \sqrt{5}}{k-1} \, ,
    \end{align*}
    which proves Eq. (\ref{final_theorem_1}), given the fact that $2\sqrt{2} + \sqrt{5} < 0.5 + \sqrt{2} + \sqrt{3} + \sqrt{5}$.
    
    Now, suppose that $l \ge 4$. By using Theorem \ref{energy_approx} once again, we directly conclude that
    \begin{align}\label{almost_done_1}
        \dfrac{E(d(l, k))}{(k-1)^{l-1/2}} &< \sum_{j = 0}^{l-1} f_j (k-1)^{-j} + 2(k-1)^{-(l-1)} \, ,\\
        \label{almost_done_2}\dfrac{E(d(l, k))}{(k-1)^{l-1/2}} &> \sum_{j = 0}^{l-1} f_j (k-1)^{-j} - 2.4(k-1)^{-(l-1)} \, .
    \end{align}
    From Eq.\ (\ref{almost_done_2}), we further obtain
    \begin{align*}
        \dfrac{E(d(l, k))}{(k-1)^{l-1/2}} &> \sum_{j = 0}^{2} f_j (k-1)^{-j} - 2.4(k-1)^{-(l-1)}\\
        &= 2 + \dfrac{2 \sqrt{2}}{k-1} + \dfrac{2 \sqrt{5} - 2}{(k-1)^2} - 2.4(k-1)^{-(l-1)}\\
        &> 2 + \dfrac{2 \sqrt{2}}{k-1} + \dfrac{2 \sqrt{5} - 2}{(k-1)^2} - 2.4(k-1)^{-2}\\
        &= 2 + \dfrac{2 \sqrt{2}}{k-1} + \dfrac{2 \sqrt{5} - 4.4}{(k-1)^2} \, ,
    \end{align*}
    which completes our entire proof of Eq.\ (\ref{final_theorem_2}), given the fact that $2 \sqrt{5} > 4.4$.
    
    From Eq.\ (\ref{almost_done_1}), we have
    \begin{align*}
        \dfrac{E(d(l, k))}{(k-1)^{l-1/2}} &< \sum_{j = 0}^{+\infty} f_j (k-1)^{-j} + 2(k-1)^{-(l-1)}\\
        &\le \sum_{j = 0}^{+\infty} f_j (k-1)^{-j} + 2(k-1)^{-3}\\
        &< 2 + \dfrac{2 \sqrt{2}}{k-1} + \dfrac{2 \sqrt{5} - 2}{(k-1)^2} + \dfrac{2}{(k-1)^3} + \sum_{j = 3}^{\infty} f_j (k-1)^{-j} \, .
    \end{align*}
    Here, it is important to notice that $\sup\limits_{j \ge 3}f_j = f_3$ due to the results obtained in Lemma \ref{cool_sequences}. This implies
    \begin{align*}
        \dfrac{E(d(l, k))}{(k-1)^{l-1/2}} &< 2 + \dfrac{2 \sqrt{2}}{k-1} + \dfrac{2 \sqrt{5} - 2}{(k-1)^2} + \dfrac{2}{(k-1)^3} + \sum_{j = 3}^{\infty} f_3 (k-1)^{-j}\\
        &= 2 + \dfrac{2 \sqrt{2}}{k-1} + \dfrac{2 \sqrt{5} - 2}{(k-1)^2} + \dfrac{2}{(k-1)^3} + \dfrac{f_3}{(k-1)^3} \sum_{j = 0}^{\infty} (k-1)^{-j}\\
        &= 2 + \dfrac{2 \sqrt{2}}{k-1} + \dfrac{2 \sqrt{5} - 2}{(k-1)^2} + \dfrac{2}{(k-1)^3} + \dfrac{f_3}{(k-1)^3} \cdot \dfrac{1}{1-\frac{1}{k-1}}\\
        &= 2 + \dfrac{2 \sqrt{2}}{k-1} + \dfrac{2 \sqrt{5} - 2}{(k-1)^2} + \dfrac{2}{(k-1)^3} + \dfrac{f_3}{(k-2)(k-1)^2} \, .
    \end{align*}
    It is clear that
    \begin{align*}
        \dfrac{1}{(k-1)^2} &\le \dfrac{1}{2} \cdot \dfrac{1}{k-1} \, ,\\
        \dfrac{1}{(k-1)^3} &\le \dfrac{1}{4} \cdot \dfrac{1}{k-1} \, ,\\
        \dfrac{1}{(k-2)(k-1)^2} &\le \dfrac{1}{2} \cdot \dfrac{1}{k-1} \, ,
    \end{align*}
    as well was
    \begin{align*}
        f_3 &= 2 \cot\left( \dfrac{\pi}{12} \right) - 2 \cot \left( \dfrac{\pi}{8} \right)\\
        &= 2 (2 + \sqrt{3}) - 2(\sqrt{2} + 1)\\
        &= 2 - 2 \sqrt{2} + 2 \sqrt{3} \, .
    \end{align*}
    We then obtain
    \begin{align*}
        \dfrac{E(d(l, k))}{(k-1)^{l-1/2}} &< 2 + \dfrac{2 \sqrt{2}}{k-1} + \dfrac{\sqrt{5} - 1}{k-1} + \dfrac{0.5}{k-1} + \dfrac{1 - \sqrt{2} + \sqrt{3}}{k-1}\\
        &= 2 + \dfrac{2\sqrt{2} + \sqrt{5} - 1 + 0.5 + 1 - \sqrt{2} + \sqrt{3}}{k-1}\\
        &= 2 + \dfrac{0.5 + \sqrt{2} + \sqrt{3} + \sqrt{5}}{k-1} \, ,
    \end{align*}
    thereby proving Eq.\ (\ref{final_theorem_1}) for $l \ge 4$.
\end{proof}
\begin{corollary}
    We have
    \begin{align*}
        E(d(l, k)) = \Theta((k-1)^{l-1/2}) \qquad \mbox{as $(l, k) \to \infty$ over $\mathbb{N} \times (\mathbb{N} \setminus \{1, 2\})$} \, .
    \end{align*}
\end{corollary}
\begin{proof}
    For $l = 1$ and $k \ge 3$, we have $E(d(2, k)) = 2 \sqrt{k}$, making it straightforward to conclude that
    \begin{align*}
        \dfrac{E(d(2, k))}{(k-1)^{1/2}} &\le 2 \sqrt{\dfrac{3}{2}} \, ,\\
        \dfrac{E(d(2, k))}{(k-1)^{1/2}} &> 2 \, ,
    \end{align*}
    i.e.\ $\dfrac{E(d(2, k))}{(k-1)^{1/2}} \in \left(2, \sqrt{6}\right]$. For $l \ge 2$ and $k \ge 3$ we can use Theorem \ref{best_approx} to obtain
    \begin{align*}
        \dfrac{E(d(l, k))}{(k-1)^{l-1/2}} &< 2 + \dfrac{0.5 + \sqrt{2} + \sqrt{3} + \sqrt{5}}{k-1} \, ,\\
        \dfrac{E(d(l, k))}{(k-1)^{l-1/2}} &> 2 + \dfrac{2 \sqrt{2}}{k-1} \, ,
    \end{align*}
    which further implies
    \begin{align*}
        \dfrac{E(d(l, k))}{(k-1)^{l-1/2}} &< 2 + \dfrac{0.5 + \sqrt{2} + \sqrt{3} + \sqrt{5}}{2} \, ,\\
        \dfrac{E(d(l, k))}{(k-1)^{l-1/2}} &> 2 \, .
    \end{align*}
    Since $2 + \dfrac{0.5+\sqrt{2}+\sqrt{3}+\sqrt{5}}{2} > \sqrt{6}$, we conclude that for all $l \ge 1$ and $k \ge 3$, the given inequality must hold:
    $$ 2 < \dfrac{E(d(l, k))}{(k-1)^{l-1/2}} < \dfrac{4.5+\sqrt{2}+\sqrt{3}+\sqrt{5}}{2} \, .$$
    The theorem statement follows directly.
\end{proof}

\section*{Conflict of interest}

The authors declare that they have no conflict of interest.

\end{document}